\documentclass[reqno,12pt, a4paper]{amsart}

\usepackage[a4paper,left=25mm,right=25mm,top=30mm,bottom=30mm,marginpar=25mm]{geometry} 
\usepackage{amsmath}
\usepackage{amssymb}
\usepackage{amsthm}
\usepackage{amscd}
\usepackage[ansinew]{inputenc}
\usepackage{cite}
\usepackage{bbm}
\usepackage{color}
\usepackage[english=american]{csquotes}
\usepackage[final]{graphicx}

\usepackage[plainpages=false,pdfpagelabels,backref=page,citecolor=red]{hyperref}
\usepackage{xcolor}
\hypersetup{
colorlinks,
linkcolor={cyan!90!black},
citecolor={magenta},
urlcolor={green!40!black}
}
\definecolor{bole}{rgb}{0.47, 0.27, 0.23}
\definecolor{burgundy}{rgb}{0.5, 0.0, 0.13}
\graphicspath{{Pictures/}}

\numberwithin{equation}{section}

\newtheoremstyle{thmlemcorr}{10pt}{10pt}{\itshape}{}{\bfseries}{.}{10pt}{{\thmname{#1}\thmnumber{ #2}\thmnote{ (#3)}}}
\newtheoremstyle{thmlemcorr*}{10pt}{10pt}{\itshape}{}{\bfseries}{.}\newline{{\thmname{#1}\thmnumber{ #2}\thmnote{ (#3)}}}
\newtheoremstyle{defi}{10pt}{10pt}{\itshape}{}{\bfseries}{.}{10pt}{{\thmname{#1}\thmnumber{ #2}\thmnote{ (#3)}}}
\newtheoremstyle{remexample}{10pt}{10pt}{}{}{\bfseries}{.}{10pt}{{\thmname{#1}\thmnumber{ #2}\thmnote{ (#3)}}}
\newtheoremstyle{ass}{10pt}{10pt}{}{}{\bfseries}{.}{10pt}{{\thmname{#1}\thmnumber{ A#2}\thmnote{ (#3)}}}

\theoremstyle{thmlemcorr}
\newtheorem{theorem}{Theorem}
\numberwithin{theorem}{section}
\newtheorem{lemma}[theorem]{Lemma}
\newtheorem{corollary}[theorem]{Corollary}
\newtheorem{proposition}[theorem]{Proposition}

\theoremstyle{thmlemcorr*}
\newtheorem{theorem*}{Theorem}
\newtheorem{lemma*}[theorem]{Lemma}
\newtheorem{corollary*}[theorem]{Corollary}
\newtheorem{proposition*}[theorem]{Proposition}
\newtheorem{problem*}[theorem]{Problem}
\newtheorem{conjecture*}[theorem]{Conjecture}

\theoremstyle{defi}
\newtheorem{definition}[theorem]{Definition}

\theoremstyle{remexample}
\newtheorem{remark}[theorem]{Remark}

\theoremstyle{ass}

\newcommand{\Dcal}{\mathcal{D}}

\newcommand{\Fcal}{\mathcal{F}}

\newcommand{\Hcal}{\mathcal{H}}
\newcommand{\Ical}{\mathcal{I}}

\newcommand{\Lcal}{\mathcal{L}}

\newcommand{\Ocal}{\mathcal{O}}

\newcommand{\Ucal}{\mathcal{U}}

\newcommand{\Sbb}{\mathbb{S}}

\DeclareMathOperator{\essinf}{ess\,inf}
\DeclareMathOperator{\esssup}{ess\,sup}

\DeclareMathOperator{\dist}{dist}

\DeclareMathOperator{\osc}{osc}

\newcommand{\norm}[1]{\|#1\|}

\newcommand{\normB}[1]{\Bigl\|#1\Bigr\|}
\newcommand{\normBB}[1]{\biggl\|#1\biggr\|}

\newcommand{\abs}[1]{|#1|}

\newcommand{\absb}[1]{\bigl|#1\bigr|}
\newcommand{\absB}[1]{\Bigl|#1\Bigr|}

\newcommand{\dprb}[1]{\bigl\langle #1 \bigr\rangle}

\newcommand{\dd}{\;\mathrm{d}}

\newcommand{\N}{\mathbb{N}}
\newcommand{\R}{\mathbb{R}}

\newcommand{\Z}{\mathbb{Z}}
\newcommand{\loc}{\mathrm{loc}}

\newcommand{\toweak}{\rightharpoonup}

\def\Xint#1{\mathchoice 
{\XXint\displaystyle\textstyle{#1}}%
{\XXint\textstyle\scriptstyle{#1}}%
{\XXint\scriptstyle\scriptscriptstyle{#1}}%
{\XXint\scriptscriptstyle\scriptscriptstyle{#1}}%
\!\int} 
\def\XXint#1#2#3{{\setbox0=\hbox{$#1{#2#3}{\int}$} 
\vcenter{\hbox{$#2#3$}}\kern-.5\wd0}} 
\def\dashint{\,\Xint-}

\renewcommand{\epsilon}{\varepsilon}
\renewcommand{\phi}{\varphi}

\title[Morrey extremals]{ Extremal functions for a fractional Morrey inequality: Symmetry properties and limit at infinity}

\author{Alireza Tavakoli}
\address{Mathematical Institute, KTH Royal Institute of Technology }
\email{alirezat@kth.se}

\subjclass[2010]{35R09, 35R11, 31B25, 46E35}

\begin{document}

\begin{abstract}
In a series of articles, Ryan Hynd and Francis Seuffert have studied extremal functions for the Morrey inequality. Building upon their work, we study the extremals of a Morrey-type inequality for fractional Sobolev spaces. We verify a few of the results in the spirit of Hynd and Seuffert concerning the symmetry of extremals and their limit at infinity.

\noindent\textsc{Keywords: fractional Sobolev spaces, H\"older spaces, Morrey's inequality, fractional $p$-Laplacian, Perron solutions}

\vspace{8pt}

\end{abstract}

\maketitle



\section{Introduction}

We consider the following fractional Sobolev class $\Dcal^{s,p}(\R^n)$
$$\Dcal^{s,p}(\R^n) := \left\{ u \in L^1_{loc}(\R^n) \quad : \quad  \normBB{\frac{u(x)-u(y)}{\abs{x-y}^{\frac{n}{p}+s}}}_{L^p(\R^n \times \R^n)}  < \infty \right\}. $$
Whenever $sp>n$, functions in this class have a continuous version and the following Morrey-type inequality holds
\[
[u]_{C^{s-\frac{n}{p}}(\R^n)}:= \sup_{x \neq y}\frac{\abs{u(x)-u(y)}}{\abs{x-y}^{s-\frac{n}{p}}} \leq C(n,s,p) [u]_{W^{s,p}(\R^{n})} \, ,
\]
where
$$[u]_{W^{s,p}(\R^n)}:= \normBB{\frac{u(x)-u(y)}{\abs{x-y}^{\frac{n}{p}+s}}}_{L^p(\R^n \times \R^n)} .$$
In this article, we always work with this continuous version.  The earliest proof of this inequality that we are aware of is due to Peetre \cite{P}. Our main focus is to study the equality case in the sharp inequality
\begin{equation}\label{eq:morrey-ineq}
[u]_{C^{s-\frac{n}{p}}(\R^n)} \leq C_\star [u]_{W^{s,p}(\R^{n})} \, ,
\end{equation}
where $C_\star$ is the best constant for the inequality. In particular, we establish some properties of the functions achieving equality in \eqref{eq:morrey-ineq}, which we call Morrey extremals. 

The seminorms $[u]_{C^{s- \frac{n}{p}}(\R^n)}$ and $[u]_{W^{s,p}(\R^n)}$ are invariant under the following transformations
\begin{itemize}
  \item $u(x) \to -u(x)$.
  \item $u(x) \to u(x)+c$
  \item $u(x) \to \lambda^{\frac{n}{p}-s} u(\lambda x)$, for $\lambda >0$. 
  \item $u(x) \to u(x+a)$. 
  \item $u(x) \to u(Ox)$, for $O \in \Ocal(\R^n)$, an orthogonal transformation of $\R^n$.
\end{itemize}
While applying the combination of these transformations allows us to generate new extremal functions for \eqref{eq:morrey-ineq} from an existing one, not all of these transformations lead to new extremals. The Morrey extremals exhibit certain symmetry properties. In \cite{Hynd} and \cite{Hynd0} it has been shown among other things that a nonconstant extremal for the inequality
\begin{equation}\label{eq:Morrey-ineq-local}
[u]_{C^{1-\frac{n}{p}}(\R^n)} \leq \Lambda(n,p) \left( \int_{\R^n} \abs{Du}^p \right)^{\frac{1}{p}}
\end{equation}
(due to Morrey \cite{M}) exists, and up to translation, rotation, dilation, and multiplication by a constant satisfies
\begin{itemize}
    \item[(i)] $-u(- \mathrm{e}_n)=u(\mathrm{e}_n) = 1$ and $\abs{u(x)}\leq 1$.
    \item[(ii)] $-\Delta_p u = \frac{2^{n-1}}{\Lambda(n,p)^p} \left( \delta_{\mathrm{e}_n} - \delta_{-\mathrm{e}_n} \right) $.
    \item[(iii)] $u$ is symmetric with respect to rotations that fix the $x_n$ axis.
    \item[(iv)] $u$ is anti-symmetric with respect to the $x_n$ variable.
    \item[(v)] $u$ is positive in the half space $\lbrace x \in \R^n \; : \; x\cdot \mathrm{e}_n >0 \rbrace$. 
\end{itemize}
Furthermore, in \cite{Hynd1}, it has been shown that extremal functions for \eqref{eq:Morrey-ineq-local} possess a limit at infinity. After establishing the existence of an extremal for \eqref{eq:morrey-ineq}, our first objective is to reproduce properties $(i)-(iv)$ in the fractional setting. These results are presented in sections \ref{sec:3}, and \ref{sec:sym}. The proofs are straightforward adaptations of certain arguments in \cite{Hynd} and \cite{Hynd0}.  A stability property of \eqref{eq:Morrey-ineq-local} is proved in \cite[Corollary 6.3]{Hynd}. The same argument adapts to the fractional setting, this is presented in Section \ref{sec:4}. 

Our subsequent task is to demonstrate that in dimensions greater than or equal to two, the extremals for \eqref{eq:morrey-ineq} have a limit at infinity. Our approach differs from that of \cite{Hynd1}, and our result is also weaker. Specifically, in \cite{Hynd1}, they establish the existence of a limit at infinity for $p$-harmonic functions in an exterior of a ball when $p>n$. An essential element in their argument is an observation made by Serrin in \cite{S} regarding non-removable singularities of $p$-harmonic functions in punctured domains. Serrin's observation states that if $p>n$, then a $p$-harmonic function in a punctured ball is continuous and satisfies the $p$-Laplace equation with delta Dirac as a right-hand side, that is $\Delta_p u= K \delta$ for some constant $K$. Furthermore, in \cite{Hynd1}, a refinement of this observation by Kichenassamy and V\'{e}ron in \cite{KV} has also been utilized. We have not been able to reproduce these results in the fractional setting. Instead, we make use of the anti-symmetry of the Morrey extremals and implement an idea of Bj\"orn in \cite{Bj} to overcome this difficulty. This is done in Section \ref{sec:6}. At the end of the article with the aid of a maximum principle for anti-symmetric functions, we also prove property $(v)$ in the fractional setting. We could only demonstrate this in dimensions greater or equal to two since our argument uses the existence of a limit at infinity for the extremals. Although we suspect that such property is also true in dimension one, we can not prove this as of right now.

Here we seize the opportunity to mention some relevant works. In \cite{HLL} they have established a decay rate for the extremal functions of \eqref{eq:Morrey-ineq-local}. In the preprint \cite{BPZ} they have shown that $\underset{p \to \infty}{\lim} \Lambda(n,p)^{\frac{1}{p}} = 1$. Where $\Lambda(n,p)$ is the best constant in \eqref{eq:Morrey-ineq-local}. For some studies of a relevant Morrey-Sobolev inequality with $L^\infty$ norm instead of the H\"older seminorm in the left-hand side we refer to \cite{EP,FM,HL}.

Finally let us mention the recent study addressing the the inequality \eqref{eq:morrey-ineq} and its extremals \cite{BPS}. Their work contains various regularity properties of the extremals and they focus on obtaining asymptotic behavior of the sharp constant in \eqref{eq:morrey-ineq} as parameters approach certain limits: $s\to 1$, $s \to \frac{n}{p}$, and $p\to \infty$. 
\subsection{Acknowledgements} The author warmly thanks Erik Lindgren for introducing the problem, proofreading this paper, for his helpful comments, and for long hours of fruitful discussions. I would also like to thank Lorenzo Brasco, Francesca Prinari, and Firoj Sk for sharing a draft version of their paper \cite{BPS}.

During the development of parts of this paper, I have been a Ph.D. student at Uppsala University. In particular, I wish to express my gratitude to the Department of Mathematics at Uppsala University for its warm and hospitable research environment

Parts of this work were done while I was participating in the program geometric aspects of nonlinear partial differential equations at Mittag-Leffler Institute in Djursholm, Sweden during the fall of 2022. The research program is supported by Swedish Research Council grant no. 2016-06596 

\section{Preliminaries}
In this section, after establishing some notation, we recall a proof of a fractional Morrey-type inequality for regional Sobolev seminorms. Additionally, we introduce the concept of weak solutions for the fractional $p$-Laplace equation and review some key properties of these solutions.
\subsection{Notation}
We define the monotone function $J_p:\R \to \R$, for $1<p<\infty$ by
$$J_p(a)=\abs{a}^{p-2}a.$$
We denote by $B(x,r)$, the ball of radius $r$ with center at $x$.\\
Let $\Omega$ be an open subset of $\R^n$. For any function $u \in L^1(\Omega)$, we use the following notations for the average of $u$ over $\Omega$. 
$$\dprb{u}_{\Omega} := \dashint_{\Omega} u(z) \dd z := \frac{1}{\abs{\Omega}}\int_{\Omega} u(z) \dd z.$$
Let $0<s<1$ and $1<p<\infty$. We introduce the following class of functions
\[
\Dcal^{s,p}(\Omega):= \left\lbrace u \in L^1_{\loc}(\Omega) \; : \;  \normBB{\frac{u(x)-u(y)}{\abs{x-y}^{\frac{n}{p}+s}}}_{L^p(\Omega\times \Omega)} < \infty  \right\rbrace.
\]
We use the following notation for the fractional Sobolev seminorm also known as the Aronszajn-Gagliardo-Slobodeckij seminorm
$$[u]_{W^{s,p}(\Omega)}:= \normBB{\frac{u(x)-u(y)}{\abs{x-y}^{\frac{n}{p}+s}}}_{L^p(\Omega\times \Omega)} = \left( \iint_{\Omega \times \Omega} \frac{\abs{u(x)-u(y)}^p}{\abs{x-y}^{n+sp}} \dd x \dd y \right)^{\frac{1}{p}}.$$
The following H\"older seminorm can be viewed as a limiting case of this fractional Sobolev seminorm as $p$ tends to infinity. 
$$[u]_{C^s(\Omega)}:= \underset{x\neq y \in \Omega}{\esssup} \, \frac{\abs{u(x)-u(y)}}{\abs{x-y}^{s}}.$$
The fractional Sobolev norm is defined as follows
\[
\norm{u}_{W^{s,p}(\Omega)} := \norm{u}_{L^p(\Omega)} + [u]_{W^{s,p}(\Omega)}
\]
The Banach space $W^{s,p}(\Omega)$ is defined as the space of measurable functions $u$ such that $\norm{u}_{W^{s,p}(\Omega)} <\infty$. We also need the definition of \textit{tail space}
\[
L^{q}_{\alpha}(\R^n):= \left\lbrace u \in L^{q}_{\loc}(\R^n) \, : \, \int_{\R^n} \frac{\abs{u(x)}^{q}}{(1+ \abs{x})^{n+\alpha}} \dd x < \infty  \right \rbrace.
\]
We have to remark that class $\Dcal^{s,p}(\R^n)\cap C^{s-\frac{n}{p}}_{\loc}(\R^n)$ does not coincide with the completion of $C^\infty_c(\R^n)$ with respect to the $W^{s,p}(\R^n)$ seminorm. One has to factor out the constant functions. See \cite{BGV} for a characterization of this space.
\subsection{Morrey estimate}
The following Morrey-type estimate is essentially contained in \cite{hitch}. See \cite{Simon} for an earlier appearance of it in dimension $n=1$. For the sake of completeness, we include a proof of this.
\begin{proposition}\label{prop:Morrey-estimate}
Let $u \in \mathcal{D}^{s,p}(B(x_0,R))$. Then $u$ has a continuous version in $B(x_0,R)$ and for any $x,y \in B(x_0,R)$ there holds
\begin{equation}\label{eq:regional-morrey}
    \abs{u(x)- u(y)} \leq C r^{s-\frac{n}{p}} [u]_{W^{s,p}\left(B(x_0,R) \right)},
\end{equation}
where $r= \abs{x-y}$ and $C$ is a constant that depends on $n,s,$ and $p$. In particular,
\begin{equation}\label{eq:regional-Morrey-2}
    \frac{\abs{u(x)- u(y)}}{\abs{x-y}^{s-\frac{n}{p}}} \leq C [u]_{W^{s,p}\left( \frac{x+y}{2}, \frac{\abs{x-y}}{2}\right)}.
\end{equation}
\end{proposition}

First, we need the following lemma due to Campanato \cite{Cam}, see also \cite{Meye}. We also include a proof for the convenience of the reader.
\begin{lemma}\label{lm:campanto}
Assume that $0<\alpha<1$ and $u \in L^1\left(B(x_0,R )\right)$. For any point $\xi \in B\left(x_0, R \right) $ define 
\[
D(\xi,\rho):= B(\xi, \rho) \cap B\left( x_0,R) \right).
\]
Let $A$ be the set of Lebesgue points of $u$. Then for $ x,y \in B(x_0, R) \cap A$,
\begin{equation}\label{eq:campanato-holder-estimate}
    \frac{\abs{u(x)-u(y)}}{\abs{x-y}^\alpha} \leq C \sup_{\xi \in B(x_0,R) ,\; 0<\rho \leq 2R } \Bigl \lbrace \rho^{-\alpha}  \dashint_{D(\xi,\rho)} \left |u(z) - \dprb{u}_{D(\xi,\rho)} \right| \dd z  \Bigr \rbrace, 
\end{equation}
where $C$ is a constant that depends only on the dimension. Moreover, if the right-hand side of \eqref{eq:campanato-holder-estimate} is finite, then $u$ has a $C^{\alpha}$ H\"older continuous version. 
\end{lemma}
\begin{proof}
Let 
$$M:= \sup_{\xi \in B(x_0,R) ,\; 0<\rho \leq 4R } \Bigl \lbrace \rho^{-\alpha}  \dashint_{D(\xi,\rho)} \left |u(z) - \dprb{u}_{D(\xi,\rho)} \right| \dd z  \Bigr \rbrace.$$
As for $\xi \in B(x_0,R)$ and $2R< \rho \leq 4R$ we have $D(\xi,\rho)=B(x_0,R)$ and $\rho^{- \alpha} < (2R)^{-\alpha}$, it is implied that
$$M= \sup_{\xi \in B(x_0,R) ,\; 0<\rho \leq 2R } \Bigl \lbrace \rho^{-\alpha}  \dashint_{D(\xi,\rho)} \left |u(z) - \dprb{u}_{D(\xi,\rho)} \right| \dd z  \Bigr \rbrace. $$
Observe that for any $\xi \in B(x_0,R)$ and $0<\rho \leq 4R$
\begin{equation}
   b(n) \abs{B(\xi,\rho)}\leq \abs{D(\xi,\rho)} \leq \abs{B(\xi,\rho)},
\end{equation}
for some dimensional constant $b(n)$. One can for example take $b(n)= 8^{-n}$. To elaborate more, for any $\rho \leq R$ the intersection $B(\xi,\rho)\cap B(x_0, R)$ contains a ball of radius $\frac{\rho}{2}$. For $R<\rho \leq 4R$, we have the obvious inclusion $D(\xi,R)\subset D(\xi,\rho)$ which contains a ball of radius $\frac{R}{2}$. Hence, $D(\xi,\rho)$ contains a ball of radius $\frac{\rho}{8}$ for all $0<\rho <4R$. 

 Suppose that $\rho < 4R$ and $h,k \in \Z$ such that $0 \leq h < k $. For any $\xi \in B(x_0,R)$ we can compute
\begin{equation}\label{eq:campanato-cauchy}
\begin{aligned}
    \absB{\dprb{u}_{D\left(\xi,\frac{\rho}{2^k}\right)}- \dprb{u}_{D\left( \xi,\frac{\rho}{2^h}\right)}} =& \left| \sum_{i=h}^{k-1} \dprb{u}_{D\left(\xi,\frac{\rho}{2^{i+1}}\right)} - \dprb{u}_{D\left(\xi,\frac{\rho}{2^i}\right)} \right| \\ 
    &\leq \sum_{i=h}^{k-1} \left| \dprb{u}_{D\left(\xi,\frac{\rho}{2^{i+1}}\right)} - \dprb{u}_{D\left(\xi,\frac{\rho}{2^i}\right)}  \right| \\
    & = \sum_{i=h}^{k-1} \left| \frac{1}{\absb{D(\xi,\frac{\rho}{2^{i+1}})}} \int_{D(\xi,\frac{\rho}{2^{i+1}})} u(z) - \dprb{u}_{D\left(\xi,\frac{\rho}{2^i}\right)} \,\dd z  \right| \\
    & \leq  \sum_{i=h}^{k-1} \frac{1}{\absb{D\left(\xi,\frac{\rho}{2^{i+1}}\right)}} \int_{D\left(\xi,\frac{\rho}{2^{i+1}}\right)} \left|{u(z) - \dprb{u}_{D\left(\xi,\frac{\rho}{2^i}\right)}}\right| \, \dd z \\
    &\leq \sum_{i=h}^{k-1} \frac{\absb{D(\xi,\frac{\rho}{2^{i}})}}{\absb{D(\xi,\frac{\rho}{2^{i+1}})}} \dashint_{D(\xi,\frac{\rho}{2^{i}})} \left|u(z) - \dprb{u}_{D\left(\xi,\frac{\rho}{2^i}\right)} \right| \dd z \\
    &\leq \sum_{i=h}^{k-1} \frac{2^n}{b(n)} \left( \frac{\rho}{2^i}\right)^\alpha M \\
    &\leq 2\cdot 16^{n}M \frac{\rho^\alpha}{2^{h\alpha}} .
    \end{aligned}
\end{equation}
This shows that the sequence $\dprb{u}_{D\left(\xi,\frac{\rho}{2^i}\right)}$ is Cauchy and the following limit exists
\[
\hat{u}(\xi):= \lim_{i \to \infty} \dprb{u}_{D\left(\xi,\frac{\rho}{2^i}\right)}.
\]
The limit is actually independent of $\rho$ as the following computation implies. Let $\xi \in B(x_0,R)$, for any $\rho_1 < \rho_2 \leq 4R$ we have
\begin{equation}\label{eq:campanato-cauchy-middle}
\begin{aligned}
\Bigl| &  \dprb{u}_{D\left(\xi,\rho_1 \right)} - \dprb{u}_{D\left(\xi,\rho_2\right)} \Bigr|\\
&= \left| \dashint_{D\left(\xi,\rho_1\right)} 
  \dprb{u}_{D\left(\xi,\rho_1\right)} -u(z)+
u(z)- \dprb{u}_{D\left(\xi,\rho_2\right)} \, \dd z \right| \\
&\leq  \dashint_{D\left(\xi,\rho_1\right)} \left| u(z)- \dprb{u}_{D\left(\xi,\rho_1\right)} \right| \, \dd z +  \dashint_{D\left(\xi,\rho_1\right)} \left| u(z)- \dprb{u}_{D\left(\xi,\rho_2 \right)} \right| \, \dd z \\
&\leq \dashint_{D\left(\xi, \rho_1 \right)} \left| u(z)- \dprb{u}_{D\left(\xi, \rho_1 \right)} \right| \, \dd z + \frac{\left| D\left(\xi, \rho_2 \right)\right|}{\left|D\left(\xi, \rho_1 \right)\right|} \dashint_{D\left(\xi, \rho_2 \right)} \left| u(z)- \dprb{u}_{D\left(\xi,  \rho_2 \right)} \right| \, \dd z \\
&\leq M\left( \rho_1^\alpha +  8^n \left(\frac{\rho_2}{\rho_1}\right)^n \rho_2^\alpha   \right) 
\end{aligned}
\end{equation}
In particular for any $\rho_2< 4R$, there exists $0\leq h \in \Z$ such that $\frac{4R}{2^{h+1}} \leq \rho_1 < \frac{4R}{2^h}$. Using \eqref{eq:campanato-cauchy} with $\rho= \frac{4R}{2^h}$ and \eqref{eq:campanato-cauchy-middle}, for any $k\in \N$ with $k >h$ we arrive at
\[
\begin{aligned}
    \Bigl|  \dprb{u}_{D\left(\xi,\rho_1 \right)} - \dprb{u}_{D\left(\xi,\frac{4R}{2^k}\right)} \Bigr| &\leq \Bigl|  \dprb{u}_{D\left(\xi,\rho_1 \right)} - \dprb{u}_{D\left(\xi,\frac{4R}{2^h}\right)} \Bigr| + \Bigl|   \dprb{u}_{D\left(\xi,\frac{4R}{2^k} \right)} - \dprb{u}_{D\left(\xi,\frac{4R}{2^h}\right)} \Bigr| \\
    & \leq M \left(  \rho_1^\alpha +  8^n \left(\frac{4R/(2^h)}{\rho_1}\right)^n \left(\frac{4R}{2^h}\right)^\alpha 
 + 2 \cdot 16^n \left(\frac{4R}{2^h}\right)^\alpha \right)\\
 & \leq M\left( \rho_1^\alpha + 16^n (2\rho_1)^\alpha + 2\cdot 16^n (2\rho_1)^\alpha   \right) \\
 & < \rho_1^\alpha M \left(1+ 2\cdot 16^n +  4\cdot 16^n \right)  
\end{aligned}
\]
Hence,
\[
\lim_{\rho \to 0} \dprb{u}_{D\left(\xi,\rho \right)} = \lim_{k \to \infty} \dprb{u}_{D\left(\xi,\frac{4R}{2^k} \right)} = \hat{u}(\xi).
\]
Moreover,
\begin{equation}\label{eq:companato-uniform-conver}
    \left| \hat{u}(\xi) - \dprb{u}_{D\left(\xi,\rho \right)}  \right| \leq C M \rho^\alpha,
\end{equation}
where $C$ is a dimensional constant. As for any fixed radius $\rho$ the functions $\xi \to \dprb{u}_{D\left(\xi,\rho \right)}$ are continuous, the uniform convergence \eqref{eq:companato-uniform-conver} implies that $\hat{u}(\xi)$ is continuous in $B(x_0,R)$.
Also notice that if $\xi$ is a Lebesgue point of $u$ , using Lebesgue's differentiation theorem,
$$u(x)= \lim_{k \to \infty} \dprb{u}_{B(x,\frac{R}{2^k})}.$$
Since for small enough values of $\rho$, $D(\xi, \rho)= B(\xi,\rho)$ we obtain
\[
u(\xi)= \hat{u}(\xi)\quad \text{for a.e.} \quad \xi \in B(x_0,R).
\]
Finally, we show that $\hat{u}$ is H\"older continuous and satisfies the bound \eqref{eq:campanato-holder-estimate}. Let $r:= \frac{\abs{x-y}}{2} $. Notice that $r <  R$.
By the triangle inequality, we have
\[
\abs{\hat{u}(x)- \hat{u}(y)} \leq \left|{\hat{u}(x) - \dprb{u}_{D(x,r)}}\right | + \left| {\hat{u}(y) - \dprb{u}_{D(y,3r)}}\right | + \left|{ \dprb{u}_{D(y,3r)}- \dprb{u}_{D(x,r)}}\right |.
\]

Using \eqref{eq:companato-uniform-conver}
\begin{equation}\label{eq:campanto2}
\begin{aligned}
    \left|{\hat{u}(x) - \dprb{u}_{D(x,r)}}\right | &\leq C M r^\alpha \\
     \left|{\hat{u}(y) - \dprb{u}_{D(y,3r)}}\right | &\leq C M (3r)^\alpha.
    \end{aligned}
\end{equation}
As for the term $\left|{ \dprb{u}_{D(y,3r)}- \dprb{u}_{D(x,r)}}\right |$, noticing that $D(x,r)\subset D(y,3r)$, we have
\begin{equation}\label{eq:campanto3}
\begin{aligned}
   \left|{\dprb{u}_{D(x,r)} - \dprb{u}_{D(y,3r)} }\right| &= \left| \dashint_{D(x,r)} u(z) - \dprb{u}_{D(y,3r)} \dd z \right|   \\
   & \leq \frac{1}{\absb{D(x,r)}} \int_{D(x,r)} \left|{u(z)- \dprb{u}_{D(y,3r)}}\right| \dd z \\
   &\leq \frac{\absb{D(y,3r)}}{\absb{D(x,r)}} \dashint_{B(y,3r)}\left|{u(z)- \dprb{u}_{D(y,3r)}}\right| \dd z \\
   &\leq \frac{3^n}{b(n)} M (3r)^\alpha.
\end{aligned}
\end{equation}
Summing the equations \eqref{eq:campanto2} and \eqref{eq:campanto3} we arrive at 
\[
\begin{aligned}
\abs{\hat{u}(x)-\hat{u}(y)}&\leq C\left( 1 + 3^\alpha  + \frac{3^{n}}{b(n)} 3^{\alpha} \right) M r^\alpha \leq C\left(1+ 3 + 3\cdot 24^n \right) M \frac{\abs{x-y}}{2}^\alpha \\
&\leq C M \abs{x-y}^\alpha,
\end{aligned}
\]
for some $C$ only depending on $n$.
\end{proof}
Now that we have Lemma \ref{lm:campanto} at hand the proof of Proposition \ref{prop:Morrey-estimate} is just an application of the Poincare inequality.
\begin{proof}[Proof of Proposition \ref{prop:Morrey-estimate}]
For any $\xi \in B(x_0,R)$ and $0<\rho \leq 2R $ similar to Lemma \ref{lm:campanto} we introduce 
$$D(\xi,\rho):= B(\xi,\rho)\cap B(x_0,R).$$
First, notice that
\[
\begin{aligned}
\int_{D(\xi,\rho)} \left|{u(z) - \dprb{u}_{D(\xi,\rho)}}\right|^p \dd z &= \int_{D(\xi,\rho)} \left| { \dashint_{D(\xi,\rho)} u(z) - u(w) \dd w}\right|^p \dd z  \\
&\leq \int_{D(\xi,\rho)} \dashint_{D(\xi,\rho)} \abs{u(z) - u(w)}^p \dd w \dd z .
\end{aligned}
\]
As for any $z,w \in D(\xi, \rho)$ we have $\abs{z-w}\leq 2\rho$,
\[ \begin{aligned}
\int_{D(\xi,\rho)} \left|{u(z) - \dprb{u}_{D(\xi,\rho)}}\right|^p \dd z &\leq \frac{(2 \rho)^{n+sp}}{\abs{D(\xi, \rho)}} \iint_{D(\xi,\rho) \times D(\xi,\rho)} \frac{\abs{u(z) - u(w)}^p}{\abs{z-w}^{n+sp}} \dd z \dd w \\
& \leq \frac{8^n \cdot 2^{n+sp}}{ \omega_n} \rho^{sp} \iint_{D(\xi,\rho) \times D(\xi, \rho)} \frac{\abs{u(z) - u(w)}^p}{\abs{z-w}^{n+sp}} \dd z \dd w .
\end{aligned}
\]
Hence,
$$ \dashint_{D(\xi, \rho)} \left|{u(z) - \dprb{u}_{D(\xi,\rho)}}\right|^p \dd z \leq C(n,s,p) \rho^{sp-n} [u]_{W^{s,p}\left(D(\xi, \rho)\right)}^p.$$
Using H\"older's inequality, we arrive at 
\begin{equation}\label{eq:campanato-pre-est}
\begin{aligned}
    \dashint_{D(\xi, \rho)} \left|{u(z) - \dprb{u}_{D(\xi,\rho)}}\right| \dd z &\leq \left( \dashint_{D(\xi,\rho)} \left| u(z)- \dprb{u}_{D(\xi,\rho)} \right|^p \right)^{\frac{1}{p}} \\
    &\leq  C \rho^{s-\frac{n}{p}} [u]_{W^{s,p}(D(\xi, \rho))}.
    \end{aligned}
\end{equation}
  Now we are in a position to use Lemma \ref{lm:campanto}. Recalling that $D(\xi,\rho)\subset B(x_0,R)$, and using \eqref{eq:campanato-pre-est} we can compute
\[
\begin{aligned}
\sup_{\xi \in B(x_0,R) ,\; 0<\rho \leq 2R } \Bigl \lbrace \rho^{\frac{n}{p}-s}  \dashint_{D(\xi,\rho)} \left |u(z) - \dprb{u}_{D(\xi,\rho)} \right| \dd z  \Bigr \rbrace 
&\leq C \sup_{\xi \in B(x_0,R) ,\; 0<\rho \leq 2R } \left \lbrace  [u]_{W^{s,p}(D(\xi, \rho))} \right \rbrace \\
&\leq C [u]_{W^{s,p}(B(x_0,R))}.
\end{aligned}
\]
As $[u]_{W^{s,p}(B(x_0,R))}$ is finite by the assumption, Lemma \ref{lm:campanto} implies that $u$ has a H\"older continuous version and for any $x,y \in B(x_0,R)$
\[
\frac{\abs{u(x)-u(y)}}{\abs{x-y}^{s-\frac{n}{p}}} \leq C [u]_{W^{s,p}(B(x_0,R))}.
\]

\end{proof}
\subsection{Notions of solutions}
In this section, we are concerned with the operator
$$(-\Delta_p)^s u(x):= \mathrm{P.V.} \int_{\R^n} \frac{J_p \left(u(x)-u(y) \right)}{\abs{x-y}^{n+sp}} \dd y.$$
Let $\Omega$ be an open subset of $\R^n$. We introduce two notions of solution for the equation 
$$(-\Delta_p)^s u(x)=0.$$
Namely weak and viscosity solutions. We mainly work with the notion of viscosity solutions except in Lemma \ref{lm:equation}, where it is easier to work with weak solutions. These two notions of solutions turn out to be equivalent under some mild assumption. See \cite{KKL}.

\subsubsection{$(s,p)$-harmonic functions (viscosity solutions)}
We will need some basic properties of viscosity solutions. The definition looks different for small values of $p$. In this article, we shall deal with the range $p\geq 2$ for the most part. Thus, we will give two separate definitions.  
\begin{definition}\label{def:viscosity1}
    Suppose that $0<s<1$ and $p > \frac{2}{2-s}$. Let $\Omega$ be an open subset of $\R^n$. We say that $u:\R^n \to [-\infty , \infty]$ is a viscosity super-solution of 
    \[
    (-\Delta_p)^s u= 0 \quad \text{in}\quad \Omega,
    \]
    or simply $(s,p)$-superharmonic if the following holds:
    
    $(i)$ $u < \infty$ almost everywhere in $\R^n$, and $u > -\infty$ everywhere in $\Omega$.
    
    $(ii)$ $u$ is lower semi-continuous in $\Omega$.

    $(iii)$ $u^- \in L_{sp}^{p-1}(\R^n)$.
    
    $(iv)$ If $\phi \in C^2 \left( B(x_0,r) \right)$ for $B_r \subset \Omega$ is such that $u(x_0)=\phi(x_0)$ and 
    $$\phi(x) \leq u(x) \quad \text{for} \quad x \in B(x_0,r). $$
    Then, 
    $$\mathrm{P.V.} \int_{\R^n} \frac{J_p(w(x_0)- w(y))}{\abs{x_0-y}^{n+sp}} \dd y \geq 0,$$
    where
    \[
    w(x):= \begin{cases}
        \phi(x) \quad &\text{for} \quad x \in B(x_0,r), \\
        u(x) \quad &\text{for} \quad x \in \R^n \setminus B(x_0,r).
    \end{cases}
    \]
\end{definition}  
  A function $u$ is called $(s,p)$-subharmonic in $\Omega$ if $-u$ is $(s,p)$-superharmonic in $\Omega$. We also say that $u$ is $(s,p)$-harmonic in $\Omega$ if $u$ is both $(s,p)$-subharmonic and $(s,p)$-superharmonic in $\Omega$.

 We can treat $(s,p)$-superharmonic functions like classical supersolutions in certain situations, see \cite[Proposition 3.1]{KKL} as well as \cite[Proposition 1]{L}. The following lemma is a simple instance of this property, for which we provide a proof. 
\begin{lemma}\label{lm:viscosity-evaluation}
    Let $0<s<1$ and $p>\frac{2}{2-s}$. Assume that $u$ is $(s,p)$-superharmonic in $\Omega$. If $z_0 \in \Omega$ is such that $u(z_0)$ is a local minimum of $u$, then
    \[
    \int_{\R^n} \frac{J_p(u(z_0)- u(y))}{\abs{z_0-y}^{n+sp}} \dd y \geq 0.
    \]
\end{lemma}
\begin{proof}
    As $u(z_0)$ is a local minimum of $u$, there is a radius $r_0>0$ such that for all $x \in B(z_0,r_0)$, $u(x)\geq u(z_0)$. This means that for any $r<r_0$ the following are valid test functions for $u$.
    \[
    \phi_r (x)  := u(z_0) \quad \text{for} \quad x \in B(z_0,r).
    \]
    Therefore, defining
    \[
    w_r(x):= \begin{cases}
        \phi_r(x) \quad &\text{for} \quad x \in B(x_0,r), \\
        u(x) \quad &\text{for} \quad x \in \R^n \setminus B(x_0,r),
    \end{cases}
    \]
    we arrive at
    \[
    \begin{aligned}
    0 &\leq \mathrm{P.V.} \int_{\R^n} \frac{J_p(w(z_0)- w(y))}{\abs{z_0-y}^{n+sp}} \dd y \\
    &= \int_{\R^n \setminus B(z_0,r)} \frac{J_p(u(z_0)- u(y))}{\abs{z_0-y}^{n+sp}} \dd y + \mathrm{P.V.}\int_{B(z_0,r)} \frac{J_p\left((\phi_r(z_0)-\phi_r(y)\right)}{\abs{z_0-y}^{n+sp}} \dd y .
    \end{aligned}
    \]
   As $\phi_r$ is constant in $B(z_0,r)$ the second integral vanishes and we obtain
   \[
    \int_{\R^n \setminus B(z_0,r)} \frac{J_p(u(z_0)- u(y))}{\abs{z_0-y}^{n+sp}} \dd y \geq 0.
   \]
   On other hand, since $u \in L_{sp}^{p-1}(\R^n)$
   \[
    \int_{\R^n \setminus B(z_0,r)} \frac{J_p(u(z_0)- u(y))}{\abs{z_0-y}^{n+sp}} \dd y <\infty.
   \]
   Since for $x\in B(z_0,r_0)$, $u(z_0)-u(x) \leq 0$, if $r_2<r_1<r_0$
   \[
   \frac{J_p(u(z_0)- u(y))}{\abs{z_0-y}^{n+sp}} \chi_{\R^n \setminus B(z_0,r_1)} \geq  \frac{J_p(u(z_0)- u(y))}{\abs{z_0-y}^{n+sp}} \chi_{\R^n \setminus B(z_0,r_2)}.
   \]
   Hence, by the monotone convergence theorem
   \[
   \begin{aligned}
   \int_{\R^n}  \frac{J_p(u(z_0)- u(y))}{\abs{z_0-y}^{n+sp}} \dd y &= \lim_{r \to 0} \int_{\R^n \setminus B(z_0,r)} \frac{J_p(u(z_0)- u(y))}{\abs{z_0-y}^{n+sp}} \dd y \\
   &= \inf_{0<r<r_0} \int_{\R^n \setminus B(z_0,r)} \frac{J_p(u(z_0)- u(y))}{\abs{z_0-y}^{n+sp}} \dd y \geq 0.
   \end{aligned}
   \]
\end{proof}
The definition of viscosity solutions in the range $p \leq \frac{2}{2-s}$ requires more careful considerations. We need to introduce a few notations. We denote the set of critical points of a differentiable function $u: \Omega \to \R$ by
$$N_u:= \left \lbrace  x\in \Omega \, : \, \nabla u (x)=0 \right \rbrace.$$
Let $d_u(x)$ be the distance function from the set of critical points, 
$$d_u(x):= \dist (x, N_u).$$
Let $D \subset \Omega$ be an open set. We define $C^2_\beta(D)$ to be the class of $C^2$ functions  such that 
\[
\underset{x \in D}{\esssup}\left(  \frac{\min \lbrace d_u(x),1  \rbrace^{\beta-1}}{\abs{\nabla u}} , \frac{\abs{D^2u(x)}}{d_u(x)^{\beta - 2}} \right) < \infty.
\]
\begin{definition}\label{def:viscosity2}
    Suppose that $0<s<1$ and $p \leq \frac{2}{2-s}$. Let $\Omega$ be an open subset of $\R^n$. We say that $u:\R^n \to [-\infty , \infty]$ is a viscosity super-solution of 
    \[
    (-\Delta_p)^s u= 0 \quad \text{in}\quad \Omega,
    \]
    or simply $(s,p)$-superharmonic if the following holds:
    
    $(i)$ $u < \infty$ almost everywhere in $\R^n$, and $u > -\infty$ everywhere in $\Omega$.
    
    $(ii)$ $u$ is lower semi-continuous in $\Omega$.

    $(iii)$ $u_{-} \in L_{sp}^{p-1}(\R^n)$.
    
    $(iv)$ If $\phi \in C^2 \left( B(x_0,r) \right)$ for $B_r \subset \Omega$ is such that $u(x_0)=\phi(x_0)$, 
    $$\phi(x) \leq u(x) \quad \text{for} \quad x \in B(x_0,r), $$
    and either of the following holds,
    
    \quad \textbf{I.} $\nabla \phi(x_0) \neq 0$,

    \quad or 

    \quad \textbf{II.} $\nabla \phi(x_0)=0$ and $x_0$ is an isolated critical point of $\phi$, and $\phi \in C^2_{\beta}(B(x_0,r))$ for some $\beta > \frac{sp}{p-1}$.
    
    Then, 
    $$\mathrm{P.V.} \int_{\R^n} \frac{J_p(w(x_0)- w(y))}{\abs{x_0-y}^{n+sp}} \dd y \geq 0,$$
    where
    \[
    w(x):= \begin{cases}
        \phi(x) \quad &\text{for} \quad x \in B(x_0,r), \\
        u(x) \quad &\text{for} \quad x \in \R^n \setminus B(x_0,r).
    \end{cases}
    \]
\end{definition}  

\begin{remark}
    Lemma \ref{lm:viscosity-evaluation} also holds in the range $1<p \leq \frac{2}{2-s}$.
    Moreover whenever a test function exists at a point $x$ (touching from below) one can evaluate the principal value $(-\Delta_p)^s u(x)$, and it is non-negative. See \cite[Proposition 3.1]{KKL}. 
\end{remark}
The following strong maximum principle is a direct consequence of this possibility to do pointwise evaluation at a minimum and maximum point.
\begin{proposition}\label{prop:strong-maximum-principle}
    Let $\Omega$ be an open set, $0<s<1$, and $1<p< \infty$. Assume that $u$ is $(s,p)$-superharmonic in $\Omega$. If there exists $z \in \Omega$ such that 
    $$u(z)= \underset{x \in \R^n}{\essinf}\, u(x) ,$$
    Then $u$ is constant almost everywhere in $\R^n$. 
    Similarly, if $v$ is a non-constant $(s,p)$-subharmonic function in $\Omega$, then the essential supremum of $v$ over $\R^n$ can not be achieved in $\Omega$.
\end{proposition}
The following comparison principle is proved in \cite[Theorem 16]{KKP} for another definition of $(s,p)$-superharmonic functions which turns out to be equivalent to the viscosity notion that we are working with, see \cite[Theorem 1.1]{KKL}. See also the comparison principle proved in \cite[Theorem 4.1]{KKL}.
\begin{proposition}\label{prop:comparison}
    Let $\Omega$ be an open subset of $\R^n$. Assume that $u$ is an $(s,p)$-superharmonic function and $v$ is an $(s,p)$-subharmonic function in $\Omega$. Furthermore, suppose that $u \geq v$ almost everywhere in $\R^n \setminus \Omega$, and for all $x \in \partial \Omega$
    $$\liminf_{\Omega \ni y \to x} u(y) \geq \limsup_{\Omega \ni y \to x} v(y), $$
 such that both sides are not simultaneously $-\infty$ or $\infty$. Then $u\geq v $ in $\Omega$.
 \end{proposition}
Viscosity solutions especially have very good convergence properties. In particular, let $u_i$ be a sequence of $(s,p)$-harmonic functions in a domain $\Omega$. Furthermore assume that $u_i$ converges locally uniformly to $u$ in $\Omega$ and almost everywhere in $\R^n$, then $u$ is also $(s,p)$-harmonic in $\Omega$. 
\subsubsection{Local weak solutions}

\begin{definition}
    Let $\Omega$ be an open subset of $\R^n$, we say that $u \in W^{s,p}_{\loc}(\Omega) \cap L^{p-1}_{sp}(\R^n)$ is weak supersolution of 
    $$(-\Delta_p)^s u=0,  \quad \text{in} \quad \Omega, $$
    if 
    $$\iint_{\R^n \times \R^n} \frac{J_p(u(x)-u(y))(\phi(x) - \phi(y))}{\abs{x-y}^{n+sp}} \dd y \geq 0,$$
    for all non-negative $\phi \in C^\infty_c(\Omega)$
\end{definition}
When $u$ is locally bounded in $\Omega$ this notion of supersolution is equivalent to Definition \ref{def:viscosity1} and \ref{def:viscosity2}, see \cite[Theorem 1.2]{KKL}.

The following uniform H\"older estimate is proved in \cite{L} with an additional assumption of $p > \frac{1}{1-s}$ whenever $p<2$. In light of the equivalence of weak and viscosity solutions for bounded functions, one can deduce the following estimate from \cite[Theorem 1.2]{DKP}. 
\begin{theorem}\label{thm:uniform-holder}
    Let $0<s<1$ and $1<p< \infty$. Assume that $u \in L^\infty(\R^n)$ is $(s,p)$-harmonic in $B(x_0,2r)$. Then there exists $\alpha >0$ and $C>0$ both of them only depending on $s$, $p$, and $n$, such that for any $\rho \leq r$
    $$ \underset{B(x_0,\rho)}{\osc} u \leq C \left( \frac{\rho}{r} \right)^\alpha \norm{u}_{L^\infty(\R^n)}.$$
    In particular
    $$[u]_{C^\alpha(B(x_0,r))} \leq C r^{-\alpha} \norm{u}_{L^\infty(\R^n)}.$$
\end{theorem}
The following Liouville-type theorem is a direct consequence of the uniform H\"older estimate above.
\begin{proposition}\label{prop:liouville}
   Let $0<s<1$ and $1<p< \infty$. If $u$ is a bounded $(s,p)$-harmonic in the whole $\R^n$ then $u$ must be constant.
\end{proposition}
\begin{proof}
Consider two separate points $x \neq y \in \R^n$. For any $r > \abs{x-y}$, as $u$ is $(s,p)$-harmonic in $B(x,2r)$, by Theorem \ref{thm:uniform-holder}
\[
\frac{\abs{u(x)-u(y)}}{\abs{x-y}^\alpha} \leq C r^{-\alpha} \norm{u}_{L^{\infty}(\R^n)}.
\]
Letting $r$ go to infinity we arrive at 
$$\abs{u(x)-u(y)}=0.$$
\end{proof}

\section{Existence of extremals}\label{sec:3}

\begin{lemma}\label{lm:existance-extremal}
There exists a (non-constant) function $v \in \Dcal^{s,p}(\R^n)$, achieving equality case in \eqref{eq:morrey-ineq} with the best constant, that is minimizing the following ratio
$$\frac{[u]_{W^{s,p}(\R^n)}}{[u]_{C^{s- \frac{n}{p}}(\R^n)}}.$$
\end{lemma}
\begin{proof}
By invariance properties of the seminorms, we can restrict ourselves to functions having H\"older seminorm one. Define 
$$\lambda = \inf \Bigl \lbrace  [u]_{W^{s,p}(\R^n)} : u \in \Dcal^{s,p}(\R^n) , \; [u]_{C^{s-\frac{n}{p}}(\R^n)}= 1  \Bigr \rbrace .$$
Then $C_\star$, the best constant in the Morrey-type inequality \eqref{eq:morrey-ineq} is $\frac{1}{\lambda}$. Choose a minimizing sequence $(u_k)_{k \in \N}$ for which
$$
\lambda = \lim_{k \to \infty} [u_k]_{W^{s,p}(\R^n)}.
$$
Now we select $x_k,y_k \in \R^n$ with $x_k \neq y_k$ such that
$$1 = [u_k]_{C^{s-\frac{n}{p}}(\R^n)} < \frac{u_k(y_k) -u_k(x_k)}{\abs{x_k - y_k}^{s-\frac{n}{p}}}  + \frac{1}{k}.$$
We perform a translation and an orthogonal transformation and define
$$v_k(z)= \abs{x_k-y_k}^{\frac{n}{p}-s} \Bigl( u_k \bigl(\abs{x_k-y_k}O_kz+x_k\bigr) - u_k(x_k) \Bigr), $$
where $O_k$ is an orthogonal transformation such that
$$O_k \mathrm{e}_n = \frac{y_k-x_k}{\abs{x_k - y_k}}.$$
Then $v_k$ satisfies
$$[v_k]_{C^{s-\frac{n}{p}}(\R^n)}=1 \quad \text{and} \quad \lim_{k \to \infty} [v_k]_{W^{s,p}(\R^n)} = \lambda.$$
In addition, we have
$$v_k(0)=0 \quad \text{and} \quad 1- \frac{1}{k}< v_k(\mathrm{e}_n) \leq 1.$$
Using the Arzela-Ascoli theorem we obtain a subsequence of $v_{k}$ converging locally uniformly to a continuous function $v$. Since the convergence is locally uniform, we get
$$v(0)=0, \quad v(\mathrm{e}_n)=1, \quad [v]_{C^{s-\frac{n}{p}}(\R^n)} \leq 1.$$
Notice that 
$$1= \frac{v(\mathrm{e}_n)-v(0)}{\abs{\mathrm{e}_n-0}} \leq [v]_{C^{s-\frac{n}{p}}(\R^n)}. $$
Therefore, $[v]_{C^{s-\frac{n}{p}}(\R^n)} =1$.

We may rewrite the fractional Sobolev seminorm as 
$$[v_k]_{W^{s,p}(\R^n)}= \left \| \frac{v_k(x)-v_k(y)}{\abs{x-y}^{\frac{n}{p}+s}} \right \|_{L^p(\R^n \times \R^n)}.$$
Since $v_k$ have uniformly bounded seminorms, we can pass to a subsequence such that 
$$\frac{v_k(x)-v_k(y)}{\abs{x-y}^{\frac{n}{p}+s}} \toweak \tilde{v}(x,y), \quad \text{in} \quad L^p(\R^n \times \R^n). $$
On the other hand, by local uniform convergence of $v_k$ to $v$, we have the pointwise convergence 
$$\frac{v_k(x)-v_k(y)}{\abs{x-y}^{\frac{n}{p}+s}} \to \frac{v(x)-v(y)}{\abs{x-y}^{\frac{n}{p}+s}}, \quad \text{in} \quad \R^n \times \R^n \setminus \lbrace x =y \rbrace.$$
Therefore, $\tilde{v}(x,y)=  \frac{v(x)-v(y)}{\abs{x-y}^{\frac{n}{p}+s}}$ and using Fatou's lemma we obtain
\[
\begin{aligned}
\relax [v]_{W^{s,p}(\R^n)}&= \left \| \frac{v(x)-v(y)}{\abs{x-y}^{\frac{n}{p}+s}} \right \|_{L^p(\R^n \times \R^n)} \leq \liminf_{k \to \infty} \left \| \frac{v_k(x)-v_k(y)}{\abs{x-y}^{\frac{n}{p}+s}} \right \|_{L^p(\R^n \times \R^n)}\\
&= \liminf_{k \to \infty} [v_k]_{W^{s,p}(\R^n)}= \lambda.
\end{aligned}
\]
In conclusion, we have found $v \in \Dcal^{s,p}(\R^n)$ with 
$$[v]_{C^{s-\frac{n}{p}}(\R^n)}=1 ,  \qquad [v]_{W^{s,p}(\R^n)} \leq \lambda. $$
By the definition of $\lambda$, we must have $[v]_{W^{s,p}(\R^n)}=\lambda$.
\end{proof}
\begin{remark}\label{rmk:rescaling-distinct-point}
    In the proof of Lemma \ref{lm:existance-extremal} we constructed a Morrey extremal $v$, such that $v(\mathrm{e}_n)=1$, $v(0)=0$, and
    $$1=[v]_{C^{s-\frac{n}{p}}}= \frac{v(\mathrm{e}_n) - v(0)}{\abs{\mathrm{e}_n - 0}}.$$
    Given $x_0\neq y_0 \in \R^n$ and $a\neq b \in \R$, we can construct the following function
    $$u(x)=(b-a)v \left( \frac{O(x-x_0)}{\abs{y_0-x_0}} \right) + a .$$
    By construction, $u(x_0)=a$ and $u(y_0)=b$.
    Using the invariances of the H\"older and fractional Sobolev seminorms, it is straightforward to verify 
    \[
    C_\star [u]_{W^{s,p}(\R^n)}= \frac{\abs{b-a}}{\abs{y_0-x_0}^{s-\frac{n}{p}}}=\frac{\abs{u(x_0)-u(y_0)}}{\abs{x_0-y_0}^{s-\frac{n}{p}}} = [u]_{C^{s-\frac{n}{p}}(\R^n)}.
    \]
    Therefore, we have constructed an extremal that achieves the H\"older seminorm in $x_0$ and $y_0$ and has two distinct prescribed values at these points. In Section \ref{sec:sym}, we show that this information determines the extremal uniquely. 
\end{remark}
We show that for any function in the homogeneous Sobolev class $\Dcal^{s,p}(\R^n)$, the H\"older seminorm is maximized. First, we recall the following finite chain lemma from \cite{Hynd}.
\begin{lemma}\label{lm:finite-chain}
    Suppose that $R>0 $ and $x,y \in \R^n \setminus B(0,2R)$. Then there are $z_1, \ldots, z_m \in \R^n \setminus B(0,2R)$, with $m \in \lbrace 1,2,3, \ldots 7 \rbrace$ such that
    \begin{equation}\label{eq:finite-chain-abolute}
        \abs{x-z_1}, \ldots , \abs{z_i-z_{i+1}}, \ldots , \abs{z_m - y} \leq \abs{y-x}
    \end{equation}
    and
    \begin{equation}\label{eq:finite-chain-exterior}
    \begin{cases}
         B\left(\frac{x+z_1}{2},\frac{r_0}{2}\right) & \text{with} \quad r_0=\abs{x-z_1} \\
         \vdots\\
         B\left(\frac{z_i+z_{i+1}}{2},\frac{r_i}{2}\right) &  \text{with} \quad r_i=\abs{z_i-z_{i+1}} \\
         \vdots \\
         B\left(\frac{z_m + y}{2}, \frac{r_m}{2}\right) & \text{with} \quad r_m=\abs{z_m - y}
    \end{cases}
    \end{equation}
    are all subsets of $\R^n \setminus B(0,R)$
\end{lemma}

\begin{proposition}\label{prop:holdersemi_achiev}
Let $n \geq 1$, $sp >n $, and $v \in \Dcal^{s,p}(\R^n)$. Assume that $v$ is non-constant. Then there exist two points $x_0, y_0 \in \R^n$ with $x_0 \neq y_0$ such that
$$\frac{v(x_0)-v(y_0)}{\abs{x_0-y_0}^{s-\frac{n}{p}}}= [v]_{C^{s-\frac{n}{p}}(\R^n)}.$$
\end{proposition}
\begin{proof}
First, we select a pair of sequences $(x_k)_{k \in \N}, (y_k)_{k \in \N}$ such that 
$$[v]_{C^{s- \frac{n}{p}}(\R^n)} = \lim_{ k \to \infty} \frac{\abs{v(x_k)- v(y_k)}}{\abs{x_k - y_k }^{s-\frac{n}{p}}} .$$
Now we claim that
\begin{equation}\label{eq:distans-maximiz}
    \liminf_{k \to \infty} \abs{x_k - y_k} >0
\end{equation}
and
\begin{equation}\label{eq:boundedness-maxim}
    \sup_{k \in \N} \abs{x_k}, \; \sup_{k \in \N} \abs{y_k} < \infty .
\end{equation}
 It follows from \eqref{eq:boundedness-maxim} that $(x_k)_{k \in \N}$ and $(y_k)_{k \in \N}$ have convergent subsequences $x_{k_i}, \, y_{k_i}$. Due to \eqref{eq:distans-maximiz} they converge to two distinct points $x_0, y_0$. Thus, we can pass to the limit in the H\"older seminorm and conclude 
$$[v]_{C^{s- \frac{n}{p}}(\R^n)} = \lim_{i \to \infty} \Bigl\lbrace  \frac{\abs{v(x_{k_i}) - v(y_{k_i})}}{\abs{x_{k_i} - y_{k_i}}^{s-\frac{n}{p}}} \Bigr \rbrace = \frac{\abs{v(x_0)- v(y_0)}}{\abs{x_0-y_0}^{s-\frac{n}{p}}}.$$
We now argue that \eqref{eq:distans-maximiz} holds. Assume towards a contradiction that $ \underset{k \to \infty}{\lim}\, \abs{x_k-y_k}= 0$. Using \eqref{eq:regional-morrey} we arrive at
\[
\begin{aligned}
\relax [v]_{C^{s- \frac{n}{p}}(\R^n)} & = \limsup_{k \to \infty} \Bigl \lbrace \frac{\abs{v(x_k) - v(y_k)}}{\abs{x_k -y_k}^{s-\frac{n}{p}}}  \Bigr \rbrace \\
& \leq C \limsup_{k \to \infty} \, [v]_{W^{s,p}\left(B\left(\frac{x_k+y_k}{2},\frac{\abs{x_k-y_k}}{2}\right)\right)} = 0.
\end{aligned}
\]
The limit vanishes since $ \absB{B\left(\frac{x_k+y_k}{2},\frac{\abs{x_k-y_k}}{2}\right)}$ converges to zero. This contradiction concludes \eqref{eq:distans-maximiz}.

Now we turn our attention to \eqref{eq:boundedness-maxim}. We split the proof into two different cases depending on whether $n>1$ or not.\\
\textbf{Proof for $n=1$}. Suppose that $n=1$. Since the H\"older seminorm is symmetric with respect to $x$ and $y$ we may assume that $x_k \leq y_k$. For the sake of contradiction, assume that \eqref{eq:boundedness-maxim} fails. After passing to a subsequence, we end up in one of the following four possible cases: 
\[
\begin{aligned}
  \text{I. }   &x_k,\,y_k \to \infty   \\
  \text{II. }  &x_k \to x , \; y_k \to \infty  \\
  \text{III. } &x_k \to -\infty , \; y_k \to \infty \\
  \text{IV. }  &x_k \to -\infty , \; y_k \to y .
\end{aligned}
\]
\textit{Case I.}
Let $\delta > 0$. Since $\frac{\abs{v(x)- v(y)}^p}{\abs{x-y}^{n+sp}}$ is integrable on $\R \times \R$, using the monotone convergence theorem there exists $L>0$ such that
\[
[v]_{W^{s,p}(\R^n)} - [v]_{W^{s,p}([-L,L])} \leq \delta.
\]
In particular
\begin{equation}\label{eq:holder-sem-small-energy-n1}
    [v]_{W^{s,p}(\R \setminus[-L ,L])} \leq \delta. 
\end{equation}
For sufficiently large values of $k$ 
$$L < x_k \leq y_k.$$
Using Proposition \ref{prop:Morrey-estimate} together with \eqref{eq:holder-sem-small-energy-n1} we arrive at
\[
\frac{\abs{v(y_k)-v(x_k)}}{\abs{y_k-x_k}^{s-\frac{n}{p}}} 
\leq C [v]_{W^{s,p}\left( B\left( \frac{x_k+y_k}{2},   \frac{\abs{x_k-y_k}}{2}   \right) \right)} 
\leq C[v]_{W^{s,p}(\R \setminus[-L,L])} 
\leq C\delta.
\]
Here we have used
\[
B\left( \frac{x_k+y_k}{2}, \frac{\abs{x_k-y_k}}{2}  \right) \subset \R\setminus [-L,L].
\]
Note that the leftmost point in the closure of the ball is $x_k$.
Hence, for every $\epsilon >0$ we have
\[
[v]_{C^{s-\frac{n}{p}}(\R^n)}= \lim_{k \to \infty} \frac{\abs{v(y_k)-v(x_k)}}{\abs{y_k-x_k}^{s-\frac{n}{p}}} \leq \epsilon.
\]
This forces $u$ to be constant. \\
\textit{Case II.} 
Without loss of generality, we may assume that $x=0$ and $v(0)=0$. Notice that 
\[
 \lim_{k \to \infty} \frac{\abs{v(y_k)}}{(y_k)^{s-\frac{n}{p}}}= \lim_{k \to \infty} \frac{\abs{v(y_k) - v(0)}}{(y_k-0)^{s-\frac{n}{p}}}=\lim_{k \to \infty} \frac{\abs{v(y_k) - v(x_k)}}{(y_k-x_k)^{s-\frac{n}{p}}}= [v]_{C^{s-\frac{n}{p}}(\R^n)}.
\]
By passing to a subsequence, we may suppose 
$$0<2y_k \leq y_{k+1}.$$
Recall that every sequence of real numbers has a monotone subsequence. We select a monotone subsequence of $\frac{\abs{v(y_k)}}{(y_k)^{s-\frac{n}{p}}}$ and we denote it again by the same index $k$.  Since 
$$[v]_{C^{s-\frac{n}{p}}(\R^n)} \geq \frac{\abs{v(y_k)}}{(y_k)^{s-\frac{n}{p}}},$$
the monotone subsequence must be increasing, that is
$$\frac{\abs{v(y_k)}}{(y_k)^{s-\frac{n}{p}}} \leq \frac{\abs{v(y_{k+1})}}{(y_{k+1})^{s-\frac{n}{p}}}.$$
We compute
\[
\begin{aligned}
    \frac{\abs{v(y_k)-v(y_{k+1})}}{(y_{k+1}-y_k)^{s-\frac{n}{p}}} &\geq \frac{\abs{v(y_{k+1})}-\abs{v(y_k)}}{(y_{k+1}-y_k)^{s-\frac{n}{p}}} \\
    &= \frac{\abs{v(y_{k+1})}}{(y_{k+1})^{s-\frac{n}{p}}} \frac{(y_{k+1})^{s-\frac{n}{p}}}{(y_{k+1}-y_k)^{s-\frac{n}{p}}} -\frac{\abs{v(y_k)}}{(y_k)^{s-\frac{n}{p}}} \frac{(y_k)^{s-\frac{n}{p}}}{(y_{k+1}-y_k)^{s-\frac{n}{p}}}\\
    & \geq \frac{\abs{v(y_k)}}{(y_k)^{s-\frac{n}{p}}} \left( \frac{(y_{k+1})^{s-\frac{n}{p}}- (y_k)^{s-\frac{n}{p}}}{(y_{k+1}-y_k)^{s-\frac{n}{p}}} \right)\\
    &\geq  \frac{\abs{v(y_k)}}{(y_k)^{s-\frac{n}{p}}} \left( \frac{(y_{k+1})^{s-\frac{n}{p}}- (y_k)^{s-\frac{n}{p}}}{(y_{k+1})^{s-\frac{n}{p}}} \right) \\
    &\geq  \frac{\abs{v(y_k)}}{(y_k)^{s-\frac{n}{p}}} \left(1- \left(\frac{1}{2} \right )^{s-\frac{n}{p}} \right ).
\end{aligned}
\]
Now we can argue as in Case I to show that given $\epsilon >0$, for large enough values of $k$, we have
\[
 \frac{\abs{v(y_k)-v(y_{k+1})}}{(y_{k+1}-y_k)^{s-\frac{n}{p}}} \leq \epsilon.
\]
Hence,
\[
[v]_{C^{s-\frac{n}{p}}(\R^n)}=\lim_{k \to \infty}   \frac{\abs{v(y_k)}}{(y_k)^{s-\frac{n}{p}}} \leq  \left(1- \left(\frac{1}{2} \right )^{s-\frac{n}{p}} \right )^{-1}\epsilon,
\]
This implies that $v$ is constant.\\
\textit{Case III.}
By subtracting a constant from $v$ we may assume that $v(0)=0$. In addition, we may assume that 
$$x_k < 0 < y_k \qquad \text{ for all } k \in \N .$$
Notice that 
\[
\begin{aligned}
\frac{\abs{v(x_k) - v(y_k)}}{\abs{x_k - y_k}^{s-\frac{n}{p}}} &\leq \frac{\abs{v(x_k)}}{\abs{x_k - y_k}^{s-\frac{n}{p}}}  + \frac{\abs{v(y_k)}}{\abs{x_k - y_k}^{s-\frac{n}{p}}} \\
& \leq \frac{\abs{v(x_k)}}{\abs{x_k}^{s-\frac{n}{p}}} +  \frac{\abs{v(y_k)}}{\abs{y_k}^{s-\frac{n}{p}}} .
\end{aligned}
\]
With a similar argument to the one in Case II, one can show after passing to a subsequence that
\[
\frac{\abs{v(y_k)}}{\abs{y_k}^{s-\frac{n}{p}}} \leq \left( 1- \left( \frac{1}{2} \right)^{s-\frac{n}{p}} \right)^{-1} \frac{\abs{v(y_{k+1}) - v(y_k)}}{\abs{y_{k+1} - y_k}^{s-\frac{n}{p}}}
\]
and
\[
\frac{\abs{v(x_k)}}{\abs{x_k}^{s-\frac{n}{p}}} \leq \left( 1- \left( \frac{1}{2} \right)^{s-\frac{n}{p}} \right)^{-1} \frac{\abs{v(x_k) - v(x_{k+1})}}{\abs{x_k - x_{k+1}}^{s-\frac{n}{p}}}.
\]
As in Case II, this implies that $[v]_{C^{s-\frac{n}{p}}(\R^n)}$ is zero and $v$ is constant.\\
\textit{Case IV.}
This case is similar to Case II.\\
\textbf{Proof for $n\geq 2$}.

We argue towards a contradiction, assume that \eqref{eq:boundedness-maxim} does not hold. We consider two cases.
\[
\begin{aligned}
    \text{I.}& \; \limsup_{k \to \infty} \abs{x_k}=\limsup_{k \to \infty} \abs{y_k}= \infty, \\
    \text{II.} & \; \limsup_{k \to \infty} \abs{x_k} < \infty, \; \limsup_{k \to \infty} \abs{y_k}= \infty .
\end{aligned}
\]
\textit{Case I.}
\begin{equation}
    \limsup_{k \to \infty} \abs{x_k}=\limsup_{k \to \infty} \abs{y_k}= \infty.
\end{equation}
After passing to a subsequence, we may assume that 
\begin{equation}\label{eq:maximize-infty}
    \lim_{k \to \infty} \abs{x_k}=\lim_{k \to \infty} \abs{y_k}= \infty.
\end{equation}
In particular 
$$R_k:= \frac{\min{ \lbrace \abs{x_k}, \abs{y_k} \rbrace}}{2}$$
is a divergent sequence.
Using Lemma \ref{lm:finite-chain} with $R=R_k$, $x=x_k$, and $y=y_k$ we can find $z_k^1, \ldots z_k^m$ satisfying \eqref{eq:finite-chain-abolute} and \eqref{eq:finite-chain-exterior}. Now we argue as in the proof for $n=1$,  Case I. 

Let $\delta>0$, and choose $L>0$ large enough so that 
$$[v]_{W^{s,p} \left(\R^n \setminus B(0,L) \right)} < \delta .$$
For large enough values of $k$, we may assume $R_k > L$. Then the balls 
$$B \left( \frac{x_k+z_k^1}{2}, \frac{r_{0,k}}{2}  \right) , \, \ldots \,
B \left( \frac{z_k^i+z_k^{i+1}}{2}, \frac{r_{i,k}}{2}  \right) , \,\ldots  ,\,
B \left( \frac{z_k^m+y_k}{2}, \frac{r_{m,k}}{2}  \right)  $$
with $r_{i,k}$ defined as in \eqref{eq:finite-chain-exterior}, are subsets of $\R^n \setminus B(0,L)$.
By Proposition \ref{prop:Morrey-estimate} and the triangle inequality
\[
\begin{aligned}
    \abs{v(x_k) -v(y_k)} &\leq \abs{v(x_k) - v(z_k^1)} + \sum_{j=1}^{m-1} \abs{v(z_k^{j}) - v(z_k^{j+1})} + \abs{v(z_k^m) - v(y_k)} \\
    & \leq C [v]_{W^{s,p}\left(\R^n \setminus B(0,L)\right)} \left( \sum_{j=0}^{m} r_{j,k}^{s-\frac{n}{p}}  \right) 
    \leq (m+1)C [v]_{W^{s,p}\left(\R^n \setminus B(0,L)\right)} \abs{x_k-y_k}^{s-\frac{n}{p}} \\
    & \leq 8C \delta \abs{x_k-y_k}^{s- \frac{n}{p}}.
\end{aligned}
\]
Therefore, for any $\epsilon >0$
\[
[v]_{C^{s-\frac{n}{p}}(\R^n)} = \lim_{k \to \infty} \frac{\abs{v(x_k) -v(y_k)}}{\abs{x_k-y_k}^{s- \frac{n}{p}}} \leq \epsilon,
\]
and $v$ must be constant.\\
\textit{Case II.}
After passing to a subsequence we may assume that 
$$\lim_{k \to \infty } x_k= x , \quad \text{and} \quad \lim_{k \to \infty} \abs{y_k}= \infty.  $$
Without loss of generality, we may assume that $x=0$ and $v(0)=0$. Notice that 
\[
[v]_{C^{s-\frac{n}{p}}(\R^n)} = \lim_{k \to \infty} \frac{\abs{v(x_k) -v(y_k)}}{\abs{x_k-y_k}^{s- \frac{n}{p}}} = \lim_{k \to \infty} \frac{\abs{v(x) -v(y_k)}}{\abs{x-y_k}^{s- \frac{n}{p}}} = \lim_{k \to \infty} \frac{\abs{v(y_k)}}{\abs{y_k}^{s-\frac{n}{p}}}.
\]
As in the earlier case $II$ corresponding to $n=1$, by passing to a subsequence, we may assume that 
\[
\frac{\abs{v(y_k)}}{\abs{y_k}^{s- \frac{n}{p}}} \leq \frac{\abs{v(y_{k+1})}}{\abs{y_{k+1}}^{s- \frac{n}{p}}},
\]
and 
\[
0 < 2 \abs{y_k} \leq \abs{y_{k+1}}.
\]
Using the triangle inequality, we have 
\[
\begin{aligned}
    \frac{\abs{v(y_{k+1}) -v(y_k)}}{\abs{y_{k+1} - y_k}^{s-\frac{n}{p}}} &\geq  \frac{\abs{v(y_{k+1}) } -\abs{v(y_k)}}{\abs{y_{k+1} - y_k}^{s-\frac{n}{p}}} \\
    &= \frac{\abs{v(y_{k+1}) } }{\abs{y_{k+1}}^{s-\frac{n}{p}}} \frac{\abs{y_{k+1}}^{s-\frac{n}{p}}}{\abs{y_{k+1} - y_k}^{s-\frac{n}{p}}} - \frac{\abs{v(y_k)}}{\abs{y_k }^{s- \frac{n}{p}}} \frac{\abs{y_k }^{s- \frac{n}{p}}}{\abs{y_{k+1} - y_k}^{s-\frac{n}{p}}} \\
   & \geq \frac{\abs{v(y_k)}}{\abs{y_k }^{s- \frac{n}{p}}} \left( \frac{\abs{y_{k+1}}^{s-\frac{n}{p}} - \abs{y_k}^{s-\frac{n}{p}}}{\abs{y_{k+1}  - y_k}^{s-\frac{n}{p}}} \right) \\
   &\geq \frac{\abs{v(y_k)}}{\abs{y_k }^{s- \frac{n}{p}}} \left( \frac{\abs{y_{k+1}}^{s-\frac{n}{p}} - \frac{1}{2^{s-n/p}}\abs{y_{k+1}}^{s-\frac{n}{p}}}{ \left(\abs{y_{k+1}} + \abs{y_k }\right)^{s-\frac{n}{p}}  } \right)\\
   & \geq \frac{\abs{v(y_k)}}{\abs{y_k }^{s- \frac{n}{p}}}  
   \left( \frac{\abs{y_{k+1}}^{s-\frac{n}{p}}\left( 1- \frac{1}{2^{s-n/p}}\right)}{\left( \frac{3}{2} \abs{y_{k+1} } \right)^{s-\frac{n}{p}}} \right) \\
   &= \frac{\abs{v(y_k)}}{\abs{y_k }^{s- \frac{n}{p}}}  \left( \left( \frac{2}{3}\right)^{s-\frac{n}{p}} - \left(\frac{1}{3} \right)^{s-\frac{n}{p}} \right).
\end{aligned}
\]
As in case I, we can show that given any $\epsilon>0$, for $k$ large enough
\[
\frac{\abs{v(y_{k+1}) -v(y_k)}}{\abs{y_{k+1} - y_k}^{s-\frac{n}{p}}} \leq \epsilon.
\]
Hence 
\[
[v]_{C^{s-\frac{n}{p}}(\R^n)} = \lim_{k \to \infty} \frac{\abs{v(x) -v(y_k)}}{\abs{x-y_k}^{s- \frac{n}{p}}}  \leq \left( \left( \frac{2}{3}\right)^{s-\frac{n}{p}} - \left(\frac{1}{3} \right)^{s-\frac{n}{p}} \right)^{-1} \epsilon.
\]
This implies that $v$ is constant.
\end{proof}

\begin{lemma}\label{lm:equation}
Let $u$ be a Morrey extremal whose $s-\frac{n}{p}$ H\"older seminorm is attained at $x_0,y_0$. Then $u$ satisfies the following equation
\begin{equation}\label{eq:harmonicity}
   C_{\star}^{p} (-\Delta _p)^s u = \frac{J_p(u(x_0)-u(y_0))}{\abs{x_0-y_0}^{sp-n}} \Bigl(\delta_{x_0}-\delta_{y_0} \Bigr).
\end{equation}
In particular $u$ is $(s,p)$-harmonic in $\R^n \setminus \lbrace x_0, y_0 \rbrace$.
\end{lemma}
\begin{proof}
We prove that $u$ is a solution of \eqref{eq:harmonicity} in the following sense: For any $\phi \in \Dcal^{s,p}(\R^n)$ we show that 

\[
    \frac{J_p\left(u(x_0) - u(y_0)\right)\left(\phi(x_0)-\phi(y_0)\right)}{\abs{x_0-y_0}^{sp-n}}= C_\star^p \iint_{\R^n \times \R^n} \frac{J_p(u(x)-u(y))(\phi(x) - \phi(y))}{\abs{x-y}^{n+sp}} \dd x \dd y.
   \]

  Since $u$ is a Morrey extremal, we have
   \begin{equation}\label{eq:pde-seminorm-extremal}
      [u]_{C^{s-\frac{n}{p}}(\R^n)}^p= \frac{\abs{u(x_0)- u(y_0)}^p}{\abs{x_0-y_0}^{sp-n}}=C_\star^p [u]_{W^{s,p}(\R^n)}^p.
   \end{equation}
   
   Using Morrey's inequality \eqref{eq:morrey-ineq} we obtain
   \begin{equation}\label{eq:pde-morreyineq}
       \frac{\abs{u(x_0)-u(y_0) + t (\phi(x_0) - \phi(y_0))}^p}{\abs{x_0-y_0}^{sp-n}} \leq C_\star^p [u+t\phi]_{W^{s,p}(\R^n)}^p.
   \end{equation}
   Notice that $a \to \abs{a}^p$ is a convex function. In particular
   $$ pJ_p(a) h \leq \abs{a+ h}^p - \abs{a}^p .$$
   Subtracting \eqref{eq:pde-seminorm-extremal} from \eqref{eq:pde-morreyineq} and using the convexity of $a \to \abs{a}^p$ we arrive at 
   \[
   \begin{aligned}
       tp \,\frac{J_p\left(u(x_0) - u(y_0)\right)\left(\phi(x_0)-\phi(y_0)\right)}{\abs{x_0-y_0}^{sp-n}}\leq C_\star^p \left( [u+t \phi]_{W^{s,p}(\R^n)}-[u]_{W^{s,p}(\R^n)} \right).
   \end{aligned}
   \]
   Hence, for $t\geq 0$ we arrive at 
   \[
   \frac{J_p\left(u(x_0) - u(y_0)\right)\left(\phi(x_0)-\phi(y_0)\right)}{\abs{x_0-y_0}^{sp-n}}\leq C_\star^p \frac{[u+t \phi]_{W^{s,p}(\R^n)}-[u]_{W^{s,p}(\R^n)}}{tp}.
   \]
   As we let $t$ tend to zero, we obtain
   \[
   \lim_{t \to 0^+} \frac{[u+t \phi]_{W^{s,p}(\R^n)}-[u]_{W^{s,p}(\R^n)}}{tp} = \iint_{\R^n \times \R^n} \frac{J_p(u(x)-u(y))(\phi(x) - \phi(y))}{\abs{x-y}^{n+sp}} \dd x \dd y.
   \]
   Therefore, we have established
   \[
   \frac{J_p\left(u(x_0) - u(y_0)\right)\left(\phi(x_0)-\phi(y_0)\right)}{\abs{x_0-y_0}^{sp-n}}\leq C_\star^p \iint_{\R^n \times \R^n} \frac{J_p(u(x)-u(y))(\phi(x) - \phi(y))}{\abs{x-y}^{n+sp}} \dd x \dd y.
   \]
   we obtain the reverse inequality by replacing $\phi$ with $-\phi$. Hence, we arrive at 
   \[
    \frac{J_p\left(u(x_0) - u(y_0)\right)\left(\phi(x_0)-\phi(y_0)\right)}{\abs{x_0-y_0}^{sp-n}}= C_\star^p \iint_{\R^n \times \R^n} \frac{J_p(u(x)-u(y))(\phi(x) - \phi(y))}{\abs{x-y}^{n+sp}} \dd x \dd y.
   \]
   In particular if $\phi \in C^\infty_c(\R^n \setminus \lbrace x_0 ,y_0 \rbrace)$ we have
   \[
   \iint_{\R^n \times \R^n} \frac{J_p(u(x)-u(y))(\phi(x) - \phi(y))}{\abs{x-y}^{n+sp}} \dd x \dd y =0.
   \]
   Since $u$ is H\"older continuous and hence locally bounded in $\R^n$, in view of the equivalence of weak and viscosity solutions in \cite[Theorem 1.2]{KKL}, $u$ is $(s,p)$-harmonic in $\R^n \setminus \lbrace x_0,y_0 \rbrace$
\end{proof}

Now we recall Clarkson's inequality, \cite[Theorem 2]{C}
\begin{theorem}
    Let $(X, \Fcal,\mu )$ be a measure space. Let, $p> 1$ and $f,g$ be measurable functions in $L^p(X, \dd \mu)$. Then for $p\geq 2$
    \[
\normB{ \frac{f+g}{2} }_{L^p}^p +\normB{\frac{f-g}{2} }_{L^p}^p \leq \frac{1}{2} \left( \norm{f}_{L^p}^p +\norm{g}_{L^p}^p \right) ,
\]
and for $1<p \leq 2$
\[
    \normB{ \frac{f+g}{2} }_{L^p}^{\frac{p}{p-1}} +\normB{\frac{f-g}{2} }_{L^p}^{\frac{p}{p-1}} \leq  \left( \frac{1}{2} \norm{f}_{L^p}^p + \frac{1}{2} \norm{g}_{L^p}^p \right)^{\frac{1}{p-1}} .
\]
\end{theorem}

This inequality implies that for $p\geq 2$, 
\begin{equation}\label{eq:clark}
\left[\frac{u+v}{2}\right]_{W^{s,p}(\R^n)}^p + \left[\frac{u-v}{2}\right]_{W^{s,p}(\R^n)}^p \leq \frac{1}{2}\left( [u]_{W^{s,p}(\R^n)}^p+ [v]_{W^{s,p}(\R^n)}^p \right),
\end{equation}
and for $1<p\leq 2$
\begin{equation}\label{eq:clark2}
\left[\frac{u+v}{2}\right]_{W^{s,p}(\R^n)}^{\frac{p}{p-1}} + \left[\frac{u-v}{2}\right]_{W^{s,p}(\R^n)}^{\frac{p}{p-1}} \leq \left( \frac{1}{2} [u]_{W^{s,p}(\R^n)}^p+ \frac{1}{2} [v]_{W^{s,p}(\R^n)}^p \right)^{\frac{1}{p-1}}.
\end{equation}
Using the above inequality, we prove the following lemma. 
\begin{lemma}\label{lm:eq-case-energy}
    Suppose that $x_0,y_0\in \R^n$ are two distinct points and let $u\in \Dcal^{s,p}(\R^n)$ be a Morrey extremal which satisfies
    $$[u]_{C^{s-\frac{n}{p}}(\R^n)}= \frac{\abs{u(x_0)- u(y_0)}}{\abs{x_0-y_0}^{s-\frac{n}{p}}}.$$
    Assume that $v \in \Dcal^{s,p}(\R^n)$ satisfies
    $$v(x_0)=u(x_0) \quad \text{and} \quad v(y_0)=u(y_0),$$
    and
    $$[v]_{W^{s,p}(\R^n)} \leq [u]_{W^{s,p}(\R^n)}. $$
    Then $u=v$.
\end{lemma}
\begin{proof}
    Define $w := \frac{u+v}{2}$. Notice that
    \[
    w(x_0)=u(x_0) \quad \text{and} \quad w(y_0)= u(y_0).
    \]
    Hence, 
    \begin{equation}
        [w]_{C^{s-\frac{n}{p}}(\R^n)} \geq [u]_{C^{s-\frac{n}{p}}(\R^n)}.
    \end{equation}
    Using inequality \eqref{eq:morrey-ineq} and the fact that $u$ is an extremal we obtain
    \[
    \begin{aligned}
    C_\star [u]_{W^{s,p}(\R^n)}=& [u]_{C^{s-\frac{n}{p}}(\R^n)} \leq [w]_{C^{s-\frac{n}{p}}(\R^n)} \leq C_\star [w]_{W^{s,p}(\R^n)} \\
   \text{(using the triangle inequality)} \quad &\leq \frac{1}{2} C_\star \left( [u]_{W^{s,p}(\R^n)} +[v]_{W^{s,p}(\R^n)} \right) \leq C_{\star} [u]_{W^{s,p}(\R^n)}.
    \end{aligned}
    \]
    Therefore the inequalities above are all equalities and in particular 
    \begin{equation}\label{eq:pre-clark}
        [w]_{W^{s,p}(\R^n)}= [u]_{W^{s,p}(\R^n)}.
    \end{equation}
    Now if $p\geq 2$ we use Clarkson's first inequality \eqref{eq:clark} to obtain
    \[
     [w]_{W^{s,p}(\R^n)}^p + \Bigl[\frac{u-v}{2} \Bigr]_{W^{s,p}(\R^n)}^p \leq \frac{1}{2}\left(  [u]_{W^{s,p}(\R^n)}^p +  [v]_{W^{s,p}(\R^n)}^p \right) \leq  [u]_{W^{s,p}(\R^n)}^p .
    \]
    It is implied by \eqref{eq:pre-clark} that
    \[
    \Bigl[\frac{u-v}{2} \Bigr]_{W^{s,p}(\R^n)} =0.
    \]
    Which implies that $u-v=c$ for some constant $c$, due to the assumption $u(x_0)= v(x_0)$ we obtain $u=v$.

    In the case $1<p <2$, we use Clarkson's other inequality \eqref{eq:clark2} and we arrive at
    \[
    \left[ w \right]_{W^{s,p}(\R^n)}^{\frac{p}{p-1}} + \left[\frac{u-v}{2}\right]_{W^{s,p}(\R^n)}^{\frac{p}{p-1}} \leq \left( \frac{1}{2} [u]_{W^{s,p}(\R^n)}^p+ \frac{1}{2} [v]_{W^{s,p}(\R^n)}^p \right)^{\frac{1}{p-1}} \leq [u]_{W^{s,p}(\R^n)}^{\frac{p}{p-1}}.
    \]
    Using \eqref{eq:pre-clark} we obtain 
    \[
    \left[\frac{u-v}{2}\right]_{W^{s,p}(\R^n)}^{\frac{p}{p-1}}=0 ,
    \]
    which implies that $u-v$ is constant. Since $u(x_0)=v(x_0)$ we obtain $u=v$. 
\end{proof}
\begin{remark}\label{rmk:uniqueness}
    The above lemma implies a uniqueness property of the Morrey extremals. Given two distinct points $x_0\neq y_0 \in \R^n$ and values $a, b \in \R$, there exists a unique Morrey extremal $u \in \Dcal^{s,p}(\R^n)$ with the property $u(x_0)=a$, $u(y_0)= b$, and 
    $$[u]_{C^{s-\frac{n}{p}}((\R^n))}= \frac{\abs{u(x_0)-u(y_0)}}{\abs{x_0-y_0}^{s-\frac{n}{p}}}$$
\end{remark}

\section{Stability}\label{sec:4}
\begin{proposition}\label{prop:stability}
Suppose $v \in \Dcal^{s,p}(\R^n)$. then there is an extremal $u \in \Dcal^{s,p}(\R^n)$ such that
$$\bigl( \frac{C_\star}{2} \Bigr)^p [u-v]_{W^{s,p}(\R^n)}^p + [v]_{C^{s-\frac{n}{p}}(\R^n)}^p \leq C_\star^p [v]_{W^{s,p}(\R^n)}^p, $$
when $2\leq  p< \infty$ and 
\[
\Bigl( \frac{C_\star}{2}\Bigr)^{\frac{p}{p-1}} [u-v]_{W^{s,p}(\R^n)}^{\frac{p}{p-1}} + [v]_{C^{s-\frac{n}{p}}(\R^n)}^{\frac{p}{p-1}}  \leq C_\star^{\frac{p}{p-1}} [v]_{W^{s,p}(\R^n)}^{\frac{p}{p-1}},
\]
when $1< p \leq 2$.
\end{proposition}
\begin{proof}
If $v$ is constant then it is a Morrey extremal itself, therefore we assume that $v$ is not constant. Using Proposition \ref{prop:holdersemi_achiev}, there exists $x_0 \neq y_0 \in \R^n$ such that
$$[v]_{C^{s-\frac{n}{p}}(\R^n)}=\frac{\abs{v(x_0)-v(y_0)}}{\abs{x_0-y_0}^{s-\frac{n}{p}}}.$$
After a proper rescaling and rotation we can select a Morrey extremal $u$ such that $u(x_0)=v(x_0),\; u(y_0)=v(y_0)$ and 
\[
[u]_{C^{s-\frac{n}{p}}(\R^n)}= [v]_{C^{s-\frac{n}{p}}(\R^n)}
\]
Since $u$ is an extremal we have
$$[u]_{W^{s,p}(\R^n)} \leq [v]_{W^{s,p}(\R^n)} .$$
Since $u(x_0)=v(x_0) \text{ and}\; u(y_0)=v(y_0)$, 
\[
[v]_{C^{s-\frac{n}{p}}(\R^n)} \leq \left[ \frac{u+v}{2} \right]_{C^{s-\frac{n}{p}}(\R^n)}.
\]
On the other hand, using the triangle inequality 
\[
\left[ \frac{u+v}{2} \right]_{C^{s-\frac{n}{p}}(\R^n)} \leq \frac{1}{2} \left( [u]_{C^{s-\frac{n}{p}}(\R^n)} + [v]_{C^{s-\frac{n}{p}}(\R^n)}  \right) = [v]_{C^{s-\frac{n}{p}}(\R^n)}.
\]
Therefore we arrive at 
\begin{equation}\label{eq:satib-hol}
    \left[ \frac{u+v}{2} \right]_{C^{s-\frac{n}{p}}(\R^n)} =[v]_{C^{s-\frac{n}{p}}(\R^n)}.
\end{equation}

\textbf{Case I.} Suppose that $2 < p < \infty$. Using \eqref{eq:satib-hol}, Morrey's inequality \eqref{eq:morrey-ineq}, and Clarkson's first inequality \eqref{eq:clark}
\[
\begin{aligned}
\Bigl( &\frac{C_\star}{2} \Bigr)^p [u-v]_{W^{s,p}(\R^n)}^p + [v]_{C^{s-\frac{n}{p}}(\R^n)}^p = C_\star^p\left[ \frac{u-v}{2} \right]_{W^{s,p}(\R^n)}^p + \left[ \frac{u+v}{2} \right]_{C^{s-\frac{n}{p}}(\R^n)}^p \\
&\leq C_\star^p\left[ \frac{u-v}{2} \right]_{W^{s,p}(\R^n)}^p + C_\star^p\left[ \frac{u+v}{2} \right]_{W^{s,p}(\R^n)}^p \\
&\leq C_\star^p \left( \frac{1}{2} [u]_{W^{s,p}(\R^n)}^p + \frac{1}{2} [v]_{W^{s,p}(\R^n)}^p  \right) \leq C_\star^p [v]_{W^{s,p}(\R^n)}^p .
\end{aligned}
\]

\textbf{Case II.} Now suppose that $1< p \leq 2 $. The argument is similar to Case I. The only difference is that we use Clarkson's second inequality \eqref{eq:clark2}. 
\[
\begin{aligned}
   \Bigl( &\frac{C_\star}{2}\Bigr)^{\frac{p}{p-1}} [u-v]_{W^{s,p}(\R^n)}^{\frac{p}{p-1}} + [v]_{C^{s-\frac{n}{p}}(\R^n)}^{\frac{p}{p-1}} 
   = C_\star^{\frac{p}{p-1}} \left[ \frac{u-v}{2} \right]_{W^{s,p}(\R^n)}^{\frac{p}{p-1}} + \left[ \frac{u+v}{2} \right]_{C^{s-\frac{n}{p}}(\R^n)}^{\frac{p}{p-1}} \\
   & \leq C_\star^{\frac{p}{p-1}} \left[ \frac{u-v}{2} \right]_{W^{s,p}(\R^n)}^{\frac{p}{p-1}} + C_\star^{\frac{p}{p-1}}  \left[ \frac{u+v}{2} \right]_{W^{s,p}(\R^n)}^{\frac{p}{p-1}} \\
   &\leq C_\star^{\frac{p}{p-1}} \left( \frac{1}{2} [u]_{W^{s,p}(\R^n)}^p + \frac{1}{2} [v]_{W^{s,p}(\R^n)}^p  \right)^{\frac{1}{p-1}} \leq C_\star^{\frac{p}{p-1}} [v]_{W^{s,p}(\R^n)}^\frac{p}{p-1}
\end{aligned}
\]
\end{proof}
\section{Symmetry of extremals}\label{sec:sym}
In this section, we use Lemma \ref{lm:eq-case-energy} to establish the symmetry properties of the Morrey extremals.

\begin{proposition}\label{prop:uniqueness}
    Suppose that $x_0\neq y_0 \in \R^n$ and $x_1\neq y_1 \in \R^n$, and assume that $u,v \in \Dcal^{s,p}(\R^n)$ are non-constant extremals with
    \[
    [u]_{C^{s-\frac{n}{p}}(\R^n)}= \frac{u(x_0)-u(y_0)}{\abs{x_0-y_0}^{s-\frac{n}{p}}} \quad \text{ and } [v]_{C^{s-\frac{n}{p}}(\R^n)}= \frac{v(x_1)-v(y_1)}{\abs{x_1-y_1}^{s-\frac{n}{p}}}.
    \]
    Then for each orthogonal transformation of $\R^n$ which satisfies 
    $$O \left( \frac{x_0-y_0}{\abs{x_0-y_0}} \right)= \frac{x_1-y_1}{\abs{x_1-y_1}} $$
    and every $x \in \R^n$ we have
    \[
    u(x)= \frac{u(x_0)-u(y_0)}{v(x_1)-v(y_1)}\cdot \left\lbrace v \Biggl( \frac{\abs{x_1 - y_1}}{\abs{x_0-y_0}} O(x-x_0) + x_1 \Biggr) - v(x_1) \right\rbrace + u(x_0) .
    \]
\end{proposition}
\begin{proof}
    Let 
    $$\Tilde{u}(x):=\frac{u(x_0)-u(y_0)}{v(x_1)-v(y_1)}\cdot \left\lbrace v \Biggl( \frac{\abs{x_1 - y_1}}{\abs{x_0-y_0}} O(x-x_0) + x_1 \Biggr) - v(x_1) \right\rbrace + u(x_0) .$$
    This function is designed so that 
    \begin{equation}\label{eq:sym-equality-disting}
        \tilde{u}(x_0)= u(x_0) \quad \text{and } \quad \tilde{u}(y_0) = u(y_0). 
    \end{equation}
    In view of the invariances of the fractional Sobolev seminorm, we can compute
    \begin{equation}\label{sym-equality-seminorm}
    \begin{aligned}
   \relax [\tilde{u}]_{W^{s,p}(\R^n)}&= \frac{u(x_0)-u(y_0)}{v(x_1)-v(y_1)}\cdot \Biggl(\frac{\abs{x_1-y_1}}{\abs{x_0-y_0}}\Biggr)^{s-\frac{n}{p}}[v]_{W^{s,p}(\R^n)} \\
    &= \frac{[v]_{W^{s,p}(\R^n)}}{[v]_{C^{s-\frac{n}{p}}(\R^n)}}[u]_{C^{s-\frac{n}{p}}(\R^n)} \\
    &= \frac{[u]_{C^{s-\frac{n}{p}}(\R^n)}}{C_{\star}} 
     = [u]_{W^{s,p}(\R^n)} .
    \end{aligned}
    \end{equation}
    The equations \eqref{eq:sym-equality-disting} and \eqref{sym-equality-seminorm} allow us to use Lemma \ref{lm:eq-case-energy} which immediately implies 
    $$u(x)= \Tilde{u}(x).$$
\end{proof}
The following symmetry property of the extremals is a direct consequence of \ref{prop:uniqueness}.
\begin{corollary}\label{cor:symmetry}
Let $u \in \Dcal^{s,p}(\R^n)$ be a Morrey extremal with 
$$[u]_{C^{s-\frac{n}{p}}(\R^n)}= \frac{\abs{u(x_0)- u(y_0)}}{\abs{x_0-y_0}^{s-\frac{n}{p}}}  .$$
Then 
$$u(x) = u(O(x-x_0) + x_0), \quad x \in \R^n$$
for any orthogonal transformation $O$ which satisfies 
$$O(y_0-x_0)= y_0-x_0 .$$
\end{corollary}
The Morrey extremals also enjoy an anti-symmetry property as described in the next proposition.
\begin{proposition}\label{cor:antisymmetry}
    Let $u \in \Dcal^{s,p}(\R^n)$ be a non-constant Morrey extremal with 
$$[u]_{C^{s-\frac{n}{p}}(\R^n)}= \frac{\abs{u(x_0)- u(y_0)}}{\abs{x_0-y_0}^{s-\frac{n}{p}}}  .$$
Then 
\[
u\left( x- 2\frac{(x-\frac{1}{2}(x_0+y_0))\cdot(x_0-y_0) }{\abs{x_0-y_0}^2}(x_0-y_0) \right) -\left(u(x_0) + u(y_0)\right) = -u(x) 
\]
\end{proposition}
\begin{proof}
    Let
    $$v:=\left(u(x_0) + u(y_0)\right) - u\left( x- 2\frac{(x-\frac{1}{2}(x_0+y_0))\cdot(x_0-y_0) }{\abs{x_0-y_0}^2}(x_0-y_0) \right) . $$
    Since the map 
    $$x \to  x- 2\frac{(x-\frac{1}{2}(x_0+y_0))\cdot(x_0-y_0) }{\abs{x_0-y_0}^2}(x_0-y_0) $$
    is a composition of an orthogonal transformation and a translation, it is implied that 
    \[
    [v]_{W^{s,p}(\R^n)} = [u]_{W^{s,p}(\R^n)}. 
    \]
    On the other hand, by design we have
    $$v(x_0)=u(x_0) \quad \text{and} \quad v(y_0)=u(x_0).$$
    Hence, Lemma \ref{lm:eq-case-energy} implies that $u=v$.
\end{proof}
To illustrate that this is an anti-symmetry property of the extremals it is convenient to perform scaling, rotation, and translation so that $x_0=\mathrm{e}_n$, $y_0=-\mathrm{e}_n$, and $u(\mathrm{e}_n)=-u(-\mathrm{e}_n)$. Then Proposition \ref{cor:antisymmetry} states that $u$ is anti-symmetric with respect to the hyperplane $\Pi = \{ x= (x^\prime,x_n) \in \R^n :  x_n=0 \}$.

We close this section by showing that the extremal values of a Morrey extremal are achieved at the points where the H\"older seminorm is maximized.

\begin{proposition}\label{prop:pointwise-bound}
    Let $u \in \Dcal^{s,p}(\R^n)$ be a non-constant Morrey extremal with 
$$[u]_{C^{s-\frac{n}{p}}(\R^n)}= \frac{u(x_0)- u(y_0)}{\abs{x_0-y_0}^{s-\frac{n}{p}}}.$$
Then $u$ achieves its maximum and minimum at $x_0$ and $y_0$. Furthermore, we have the following strict inequality.
\[
u(y_0) < u(x) < u(x_0), \qquad \text{ for } x \in \R^n\setminus \{ x_0,y_0 \}.
\]
\end{proposition}
\begin{proof}
   Without loss of generality we can assume that $x_0= \mathrm{e}_n$, $y_0=-\mathrm{e}_n$ and  $u(\mathrm{e}_n)=-u(-\mathrm{e}_n) = 1$. 

    Now consider $w(x)= \min \{ u(x),1\}$. Using the pointwise inequality 
    \[
    \absb{\min \{ a,1\} - \min \{ b,1\} } \leq \abs{a-b},
    \]
    we arrive at 
    \begin{equation}\label{eq:cutoff-inequality}
        \abs{w(x)-w(y)} \leq \abs{u(x)-u(y)}.
    \end{equation}
    Hence,
    \[
      \iint_{\R^n \times \R^n} \frac{\abs{w(x) - w(y)}^p}{\abs{x-y}^{n+sp}} \dd x \dd y \leq  \iint_{\R^n \times \R^n} \frac{\abs{u(x) - u(y)}^p}{\abs{x-y}^{n+sp}} \dd x \dd y .
    \]
    As we know that
    $$w(\mathrm{e}_n)=u(\mathrm{e}_n) \quad \text{and } \quad w(-\mathrm{e}_n)= u(-\mathrm{e}_n),$$
    by appealing to Lemma \ref{lm:eq-case-energy} we conclude that $w$ is a Morrey extremal itself and we have $w=u$. 
    Therefore, we have established that 
    $$u(x) \leq 1, \qquad \text{for } x \in \R^n. $$
    In a similar fashion, we can argue that 
    $$u(x) \geq -1 \qquad \text{for } x \in \R^n.$$
    Finally, since $u$ is $(s,p)$-harmonic in $\R^n \setminus \lbrace \mathrm{e}_n, -\mathrm{e}_n \rbrace$, using the strong maximum principle \ref{prop:strong-maximum-principle}, the maximum, and the minimum can only be achieved in $\mathrm{e}_n$ and $-\mathrm{e}_n$.
\end{proof}

\section{Limit at infinity}\label{sec:6}
 In this section, we show that the Morrey extremals have a limit at infinity in dimensions greater than or equal to 2. Our proof is inspired by an argument of Bj\"orn in \cite{Bj} concerning the continuity of Perron solutions of certain boundary values with jump discontinuity.
\begin{theorem}\label{thm:limit-at-infty}
    Let $n \geq 2$. Suppose that $u$ is an extremal with
    $$[u]_{C^{s-\frac{n}{p}}(\R^n)}= \frac{\abs{u(x_0) - u(y_0)}}{\abs{x_0-y_0}^{s-\frac{n}{p}}}.$$
    Then 
    \[
    \lim_{x \to \infty} u(x) = \frac{1}{2} (u(x_0) + u(y_0)).
    \]
\end{theorem}
 Before starting the proof we need some properties of Perron solutions.
\subsection{Perron solutions}
Here we recall the definition and some properties of Perron solutions. We follow the same definition as in \cite{KKP}. The definition of $(s,p)$-superharmonic functions used in \cite{KKP} is based on comparison, while we use the viscosity solutions here. See Definitions \ref{def:viscosity1} and \ref{def:viscosity2}. In light of \cite[Theorem 1.1]{KKL} these two notions of solutions are equivalent.
\begin{definition}
  Let $\Omega \subset \R^n$ be an open set and assume that $g \in L_{sp}^{p-1}(\R^n)$. We define the upper Perron class of $g$, $\Ucal_g$ to be the set of all functions $v:\R^n \to [-\infty, \infty]$ such that

$(i)$ $v$ is $(s,p)$-superharmonic in $\Omega$,

$(ii)$ $v$ is bounded from below in $\Omega$,

$(iii)$ $\underset{\Omega \ni y \to x}{\liminf }\; v(y) \geq \underset{ \R^n \setminus \Omega \ni y \to x}{\limsup } g(y)$ for all $x \in \partial \Omega$,

$(iv)$ $v = g$ almost everywhere in $\R^n \setminus \Omega$.

   Furthermore, define the upper Perron solution of the complementary value $g$ to be
   $$\overline{P}g(x):= \inf_{v \in \Ucal_g} v(x).$$
\end{definition}
\begin{definition}
   Let $\Omega \subset \R^n$ be an open set and assume that $g \in L_{sp}^{p-1}(\R^n)$. We define the lower Perron class of $g$, $\Lcal_g$ to be the set of all functions $v:\R^n \to [-\infty, \infty]$ such that

$(i)$ $v$ is $(s,p)$-subharmonic in $\Omega$,

$(ii)$ $v$ is bounded from above in $\Omega$,

$(iii)$ $\underset{\Omega \ni y \to x}{\limsup}\; v(y) \leq \underset{ \R^n \setminus \Omega \ni y \to x}{\liminf} g(y)$ for all $x \in \partial \Omega$,

$(iv)$ $v = g$ almost everywhere in $\R^n \setminus \Omega$.

Furthermore, define the lower Perron solution of the complementary value $g$ to be
   $$\underline{P}g(x):= \sup_{v \in \Lcal_g} v(x).$$
\end{definition}
It follows from the comparison principle that 
$$-\infty \leq \underline{P}g \leq \overline{P}g \leq \infty $$
If we consider a bounded complement value $M_1  \leq g\leq M_2$, we have $M_1 \chi_{\Omega} + g \chi_{\R^n \setminus \Omega} \in \Lcal_g$ and $M_2\chi_{\Omega} + g \chi_{\R^n \setminus \Omega} \in \Ucal_g $, see \cite[Lemma 18]{KKP}. Hence, the classes $\Lcal_g$ and $\Ucal_g$ are nonempty and
$$M_1 \leq \underline{P}g \leq \overline{P}g \leq M_2. $$
Furthermore, for all bounded complementary values, the upper and lower Perron solutions are $(s,p)$-harmonic in $\Omega$. See \cite[Theorem 2]{KKP} as well as \cite[Theorem 22.]{LL}
Now we discuss the boundary behaviour of Perron solutions. The notion of a barrier is a classical tool to investigate whether boundary values are achieved for continuous functions. Here is the definition in the nonlocal setting.
\begin{definition}
    We say that a function $w$ is a barrier at $x_0 \in \partial \Omega$ if 

    $(i)$ $w$ is continuous in $\R^n$

    $(ii)$ $w$ is $(s,p)$-superharmonic in $\Omega$

    $(iii)$ $w(x) >0$ if $x\neq x_0$ and 
    $$\liminf_{\abs{x} \to \infty} w(x) >0.$$

    $(iv)$  $w(x_0)=0$
\end{definition}
The proof of the following proposition for continuous complement values in the whole $\R^n$ can be found in \cite[Lemma 17, Proposition 24, and Theorem 26]{LL}. An inspection of the proof reveals that the continuity assumption is only needed at the point $x_0$. 
\begin{proposition}
    Let $f:\Omega^c \to \R$ be a bounded function. Assume that $f$ is continuous at $x_0 \in \partial \Omega$, where $x_0$ admits a barrier, then 
    $$\lim_{\Omega \ni x \to x_0}\overline{P}f(x)= \lim_{\Omega \ni x \to x_0}\underline{P}f(x) = f(x_0).$$
\end{proposition}
We need a uniform variant of this property which we state in Proposition \ref{prop:barrier}. First, we construct an explicit barrier for every boundary point when $sp>n$. When $sp<n$, a computation of $(-\Delta_p)^s \abs{x}^\beta$, for a range of values of $\beta$ has been carried out in \cite[Lemma A.2]{BMS}. In particular $\abs{x}^{\frac{sp-n}{p-1}}$ is $(s,p)$-harmonic in $\R^n\setminus \lbrace 0 \rbrace$, see \cite[Theorem A.4]{BMS}. In the recent preprint \cite{DQ}, this computation has also been carried out when $sp\neq n$, furthermore it is shown in \cite{DQ} that $\log(\abs{x})$ is $(s,p)$-harmonic in $\R^n\setminus \lbrace 0 \rbrace$ whenever $sp=n$.  Here for the sake of completeness, we include a computation in the case $sp>n$.

\begin{proposition}\label{prop:barrier}
   Let $sp>n$, then the function $G(x)= \abs{x}^{\frac{sp-n}{p-1}}$ is a classical solution of 
    $$(-\Delta_p)^s G=0 \quad \text{in } \R^n \setminus \lbrace 0 \rbrace .$$
    In particular $G$ is $(s,p)$-harmonic in $\R^n \setminus \lbrace 0 \rbrace$.
\end{proposition}
\begin{proof}
    As $G$ is smooth in $\R^n\setminus \lbrace 0 \rbrace$ and $\nabla G(x) \neq 0$ for $x \neq 0$, by \cite[Lemma 3.8]{KKL} the principal value
    $$f(x)= \mathrm{P.V.} \int_{R^n} \frac{J_p(G(x)-G(y))}{\abs{x-y}^{n+sp}} \dd y,$$
    is well defined and continuous in $\R^n \setminus \lbrace 0 \rbrace$. Since $G$ is a radial function, $f$ is also radial. By a scaling argument, it is easy to see that $f$ is homogeneous of degree $-n$. 
    We split the proof into two cases depending on whether $n>1$ or not. 

    \textbf{Case $n=1$}. As $f$ is a radial and homogeneous function of degree $-1$, it is enough to evaluate the integral just at one point.  
    \[
    \begin{aligned}
        f(1)&= \lim_{\delta \to 0^+} \int_{\R\setminus (1-\delta, 1+ \delta)} \frac{J_p\left(1-\abs{y}^{\frac{sp-1}{p-1}} \right)}{\abs{1-y}^{sp+1}} \dd y \\
        &= \lim_{\delta \to 0^+} \left( \int_{(0,\infty) \setminus (1-\delta, 1+\delta)}  \frac{J_p\left(1-y^{\frac{sp-1}{p-1}}\right)}{\abs{1-y}^{sp+1}} \dd y + \int_0^\infty \frac{J_p\left(1-y^{\frac{sp-1}{p-1}}\right) }{\abs{1+y}^{sp+1}} \dd y \right).
    \end{aligned}
    \]
   Notice that for $y>0$, $\frac{1}{(1+1/y)^{sp+1}}= y^{sp+1} \frac{1}{(1+y)^{sp+1}}$. By splitting the integral and change of variables $\rho=1/y$ we obtain:
  
    \begin{equation}\label{eq:foundamental-f1}
    \begin{aligned}
   \int_0^\infty \frac{J_p\left(1-y^{\frac{sp-1}{p-1}}\right) }{\abs{1+y}^{sp+1}} \dd y &= \int_0^1 \frac{J_p\left(1-y^{\frac{sp-1}{p-1}}\right) }{\abs{1+y}^{sp+1}} \dd y + \int_1^\infty \frac{J_p\left(1-y^{\frac{sp-1}{p-1}}\right) }{\abs{1+y}^{sp+1}} \dd y \\
   & =  \int_0^1 \frac{J_p\left(1-y^{\frac{sp-1}{p-1}}\right) }{\abs{1+y}^{sp+1}} \dd y  + \int_0^1 \rho^{sp+1} \frac{J_p\left(1-(1/\rho)^{\frac{sp-1}{p-1}}\right) }{\abs{1+\rho}^{sp+1}} \rho^{-2} \dd \rho \\
   & = \int_0^1 \frac{J_p\left(1-y^{\frac{sp-1}{p-1}}\right) }{\abs{1+y}^{sp+1}} \dd y + \int_0^1 \frac{\rho^{sp-1}}{\rho^{sp-1}} \frac{J_p\left(\rho^{\frac{sp-1}{p-1}} -1\right)}{\abs{1+\rho}^{sp+1}} \dd \rho \\
   &=  \int_0^1 \frac{J_p\left(1-y^{\frac{sp-1}{p-1}}\right) + J_p \left(y^{\frac{sp-1}{p-1}} - 1 \right) }{\abs{1+y}^{sp+1}} \dd y \\
   &= 0
   \end{aligned}
    \end{equation}
   Hence,
    \[
	\begin{aligned}
        f(1)&= \lim_{\delta \to 0^+} \int_{(0,\infty) \setminus (1-\delta, 1+\delta)} \frac{J_p\left(1-y^{\frac{sp-1}{p-1}}\right)}{\abs{1-y}^{sp+1}} \dd y \\
         & = \lim_{\delta \to 0} \left( \int_0^{1-\delta} \frac{J_p\left(1-y^{\frac{sp-1}{p-1}}\right)}{\abs{1-y}^{sp+1}}  \dd y + \int_{1+\delta}^\infty \frac{J_p\left(1-y^{\frac{sp-1}{p-1}}\right)}{\abs{1-y}^{sp+1}}  \dd y  \right).
    \end{aligned}
    \]
    With a change of variable computation as in \eqref{eq:foundamental-f1} we arrive at
    \[
    \begin{aligned}
    f(1) &= \lim_{\delta \to 0} \left( \int_0^{1-\delta} \frac{J_p\left(1-y^{\frac{sp-1}{p-1}}\right)}{\abs{1-y}^{sp+1}}  \dd y -  \int_{0}^{\frac{1}{1+\delta}}  \frac{J_p\left(1-y^{\frac{sp-1}{p-1}}\right)}{\abs{1-y}^{sp+1}}  \dd y  \right) \\
    &=  - \lim_{\delta \to 0^+} \int_{1-\delta}^{\frac{1}{1+\delta}} \frac{J_p\left(1-y^{\frac{sp-1}{p-1}}\right)}{\abs{1-y}^{sp+1}}  \dd y .
    \end{aligned}
    \]
    It remains to show that 
    \begin{equation}\label{eq:foundamental-n1-lim}
    \lim_{\delta \to 0^+} \int_{1-\delta}^{\frac{1}{1+\delta}} \frac{J_p\left(1-y^{\frac{sp-1}{p-1}}\right)}{\abs{1-y}^{sp+1}}  \dd y = 0.
    \end{equation}
     Since $\frac{sp-1}{p-1}<1$, for $0<y<1$ we have 
    $$J_p\left( 1-y^{\frac{sp-1}{p-1}} \right) \leq J_p(1-y).$$
    Hence,
    \[
    \begin{aligned}
    0\leq \int_{1-\delta}^{\frac{1}{1+\delta}} \frac{J_p\left(1-y^{\frac{sp-1}{p-1}}\right)}{\abs{1-y}^{sp+1}} \dd y &\leq
     \int_{1-\delta}^{\frac{1}{1+\delta}} J_p(1-y)(1-y)^{-sp-1} \\
    &=\int_{1-\delta}^{\frac{1}{1+\delta}}  (1-y)^{p-sp-2} \dd y = \int_{\frac{\delta}{1+\delta}}^{\delta} y^{p-sp-2} \dd y\\
    &= \frac{1}{p-sp-1} \left ( - \left( \frac{\delta}{1+\delta} \right)^{p-sp-1} + \delta^{p-sp-1} \right) .
    \end{aligned}
    \]
    Where, in the last line we have assumed $p-sp-1 \neq 0$. If $p-sp-1 > 0$, then both $\delta^{p-sp-1}$ and $\left(\delta/(1+\delta) \right)^{p-sp-1}$ converge to zero as $\delta$ tends to $0$.  If $p-sp-1<0$ then this becomes
    \[
    \begin{aligned}
       \frac{-1}{p-sp-1} \left (\left( \frac{\delta}{1+\delta} \right)^{p-sp-1} - \delta^{p-sp-1} \right) = \frac{1}{sp+1-p}\delta^{p-sp} \left( \frac{(1+\delta)^{sp+1-p} -1}{\delta}  \right).
    \end{aligned}
    \]
    Notice that
    $$\lim_{\delta \to 0}  \frac{(1+\delta)^{sp+1-p} -1}{\delta}  = sp+1-p.$$
    As $p-sp>0$, we obtain
    \[
    \lim_{\delta \to 0} \frac{1}{sp+1-p}\delta^{p-sp} \left( \frac{(1+\delta)^{sp+1-p} -1}{\delta}  \right) = \lim_{\delta \to 0} \delta^{p-sp}=0.
    \]
    Finally if $p-sp-1=0$, then
    \[
    \begin{aligned}
    \int_{1-\delta}^{\frac{1}{1+\delta}}  (1-y)^{p-sp-2} \dd y &= -\ln(1-y) \Bigr|_{1-\delta}^{\frac{1}{1+\delta}}= -\left( \ln\left(\frac{\delta}{1+\delta}\right) -\ln(\delta) \right) \\
    &= \ln(1+\delta).
    \end{aligned}
    \]
    The limit of $ \ln(1+\delta)$ as delta tends to zero is also zero. Hence we have verified \eqref{eq:foundamental-n1-lim}.

    \textbf{Case $n>1$}. Similar to case $n=1$, as $f$ is radial and homogeneous of degree $-n$, it is enough to compute the integral at one point. we write the computation at $\mathrm{e}_1$. We have
    \begin{equation}\label{eq:foundamental-sol-n}
    \begin{aligned}
    f(\mathrm{e}_1) &= \lim_{\delta \to 0} \int_{\R^n \setminus B(\mathrm{e}_1, \delta)} \frac{J_p \left( 1- \abs{y}^{\frac{sp-n}{p-1}}  \right)}{\abs{ \mathrm{e}_1 - y}^{n+sp}} \dd y \\
    &= \lim_{\delta \to 0} \left( \int_{B(0,1)\setminus B(\mathrm{e}_1,\delta)} \frac{J_p \left( 1- \abs{y}^{\frac{sp-n}{p-1}}  \right)}{\abs{ \mathrm{e}_1 - y}^{n+sp}} \dd y  + \int_{B(0,1)^c \setminus B(\mathrm{e}_1,\delta)} \frac{J_p \left( 1- \abs{y}^{\frac{sp-n}{p-1}}  \right)}{\abs{ \mathrm{e}_1 - y}^{n+sp}}  \dd y \right)
    .
    \end{aligned}
    \end{equation}
    Now we make a change of variables $y= \frac{x}{\abs{x}^2}$ in the second integral. The inversion map $x \to \frac{x}{\abs{x}^2}$ maps $B(0,1)\setminus \lbrace 0 \rbrace$ into $B(0,1)^c$ and vice versa. It is also a conformal map and maps spheres that do not pass through the origin into spheres. In particular $B(\mathrm{e}_1,\delta)$ is mapped into $B \left( \frac{1}{1-\delta^2}\mathrm{e}_1,\frac{\delta}{1-\delta^2} \right)$. It is straightforward to compute the Jacobian determinant of the map:
    \[
    \det J_{y}(x) = \det \left( \frac{1}{\abs{x}^2} \mathrm{I}_n -\frac{2}{\abs{x}^4} (1,1, \ldots ,1) \otimes x \right) = \frac{1}{\abs{x}^{2n}}.
    \]
    Hence,
    \[
    \begin{aligned}
    \int_{B(0,1)^c \setminus B(\mathrm{e}_1,\delta)} \frac{J_p \left( 1- \abs{y}^{\frac{sp-n}{p-1}}  \right)}{\abs{ \mathrm{e}_1 - y}^{n+sp}}  \dd y &= \int_{\lbrace \overline{B(0,1)} \setminus \lbrace 0 \rbrace \rbrace \setminus B \left( \frac{1}{1-\delta^2}\mathrm{e}_1,\frac{\delta}{1-\delta^2} \right) } \frac{J_p \left( 1- \absB{\frac{x}{\abs{x}^2}}^{\frac{sp-n}{p-1}}  \right)}{\left| \mathrm{e}_1 - \frac{x}{\abs{x}^2}\right|^{n+sp}} \frac{1}{\abs{x}^{2n}} \dd x \\
    &=\int_{B(0,1) \setminus B \left( \frac{1}{1-\delta^2}\mathrm{e}_1,\frac{\delta}{1-\delta^2} \right)} \frac{\abs{x}^{n-sp}J_p \left( \abs{x}^\frac{sp-n}{p-1}- 1  \right)}{\left| \mathrm{e}_1 -\frac{x}{\abs{x}^2} \right|^{n+sp}} \frac{1}{\abs{x}^{2n}} \dd x \\
    &= -\int_{B(0,1) \setminus B \left( \frac{1}{1-\delta^2}\mathrm{e}_1,\frac{\delta}{1-\delta^2} \right)} \frac{J_p \left( 1- \abs{x}^{\frac{sp-n}{p-1}}  \right)}{\left| \abs{x} \mathrm{e}_1 -\frac{x}{\abs{x}} \right|^{n+sp}} \dd x
    .
    \end{aligned}
    \]
    The reflection through the hyperplane $\Sigma:= \lbrace y \in \R^n \; : \; y\cdot(\mathrm{e}_1 - \frac{x}{\abs{x}})=0 \rbrace$, sends $\abs{x} \mathrm{e}_1 -x$ to $x- \mathrm{e}_1 $. Therefore $\abs{\mathrm{e_1} - x} = \left| \abs{x} \mathrm{e}_1 -\frac{x}{\abs{x}} \right| $. Hence,
    \[
    \int_{B(0,1)^c \setminus B(\mathrm{e}_1,\delta)} \frac{J_p \left( 1- \abs{y}^{\frac{sp-n}{p-1}}  \right)}{\abs{ \mathrm{e}_1 - y}^{n+sp}}  \dd y = -\int_{B(0,1) \setminus B \left( \frac{1}{1-\delta^2}\mathrm{e}_1,\frac{\delta}{1-\delta^2} \right)} \frac{J_p \left( 1- \abs{x}^{\frac{sp-n}{p-1}}  \right)}{\left|  \mathrm{e}_1 -x\right|^{n+sp}} \dd x.
    \]
    Inserting this into \eqref{eq:foundamental-sol-n} we arrive at
    \[
    f(\mathrm{e}_1)= - \lim_{\delta \to 0} \int_{B(0,1)\cap \left \lbrace B(\mathrm{e}_1,\delta) \setminus B \left( \frac{1}{1-\delta^2}\mathrm{e}_1,\frac{\delta}{1-\delta^2} \right) \right \rbrace}  \frac{J_p \left( 1- \abs{x}^{\frac{sp-n}{p-1}}  \right)}{\left|  \mathrm{e}_1 -x\right|^{n+sp}} \dd x.
    \]
    Here we have used that $B \left( \frac{1}{1-\delta^2}\mathrm{e}_1,\frac{\delta}{1-\delta^2} \right) \cap B(0,1) \subset B(\mathrm{e}_1, \delta) \cap B(0,1)$. To see this one can argue that since $\partial B \left( \frac{1}{1-\delta^2}\mathrm{e}_1,\frac{\delta}{1-\delta^2} \right) \cap \partial B(\mathrm{e}_1, \delta) \subset \partial B(0,1)$ and  $(\frac{1}{1+\delta} \mathrm{e}_1) \in \partial B \left( \frac{1}{1-\delta^2}\mathrm{e}_1,\frac{\delta}{1-\delta^2} \right)$ belongs to $B(0,1)\cap B(\mathrm{e}_1, \delta)$ we have $B \left( \frac{1}{1-\delta^2}\mathrm{e}_1,\frac{\delta}{1-\delta^2} \right) \cap B(0,1) \subset B(\mathrm{e}_1, \delta) \cap B(0,1)$. 
    To show that the limit is zero, we observe that $B(0,1)\cap \left \lbrace B(\mathrm{e}_1,\delta) \setminus B \left( \frac{1}{1-\delta^2}\mathrm{e}_1,\frac{\delta}{1-\delta^2} \right) \right \rbrace$ is contained in the following conical ring
    \[
    E_\delta:= \left\lbrace \mathrm{e}_1 -y \; : \; \frac{y\cdot \mathrm{e}_1}{\abs{y}} >\frac{\delta}{2}  \right\rbrace \cap  \left \lbrace B(\mathrm{e}_1,\delta) \setminus  B\left(\mathrm{e}_1,\frac{\delta}{1+ \delta}\right) \right \rbrace.
    \]
    Note also that $E_\delta\subset B(0,1)$. Hence,
    \[
    \int_{B(0,1)\cap \left \lbrace B(\mathrm{e}_1,\delta) \setminus B \left( \frac{1}{1-\delta^2}\mathrm{e}_1,\frac{\delta}{1-\delta^2} \right) \right \rbrace}  \frac{J_p \left( 1- \abs{x}^{\frac{sp-n}{p-1}}  \right)}{\left|  \mathrm{e}_1 -x\right|^{n+sp}} \dd x \leq \int_{E_\delta}  \frac{J_p \left( 1- \abs{x}^{\frac{sp-n}{p-1}}  \right)}{\left|  \mathrm{e}_1 -x\right|^{n+sp}} \dd x.
    \]
    As for any $x \in E_\delta$, $\abs{x}\leq 1 $ and since $\frac{sp-n}{p-1}<1$, we have $1 - \abs{x}^{\frac{sp-n}{p-1}} \leq 1- \abs{x}$. Furthermore using the triangle inequality $1-\abs{x} \leq \abs{\mathrm{e}_1-x}$. Therefore,
    \[
    \begin{aligned}
    \int_{E_\delta} \frac{J_p \left( 1- \abs{x}^{\frac{sp-n}{p-1}}  \right)}{\left|  \mathrm{e}_1 -x\right|^{n+sp}} \dd x &\leq \int_{E_\delta} \frac{\abs{\mathrm{e}_1-x}^{p-1}}{\abs{\mathrm{e}_1-x}^{sp+n}} \dd x \\
    &= \Hcal^{n-2}(\Sbb^{n-2})\int^{\delta}_{\frac{\delta}{1+\delta}} \int_0^{\cos^{-1}(\delta/2)} r^{n-1} r^{p-sp-n-1} \sin^{n-2}{\theta} \dd \theta \dd r \\
    &\leq \frac{\pi}{2} \Hcal^{n-2}(\Sbb^{n-2}) \int_{\frac{\delta}{1+\delta}}^{\delta} r^{p-sp-2} \dd r.
    \end{aligned}
    \]
    As in the case $n=1$, the integral above converges to zero as $\delta$ decreases to zero. Hence, we have established that 
    $f(\mathrm{e}_1)=0$. 
\end{proof}
\begin{lemma}\label{lm:barrier}
    Let $sp >n$, $\Omega$ be an open set and $x_0 \in \partial \Omega$. Assume that $f : \R^n \to \R$ is a bounded function, say $\abs{f(y)} < M$ for all $y \in \R^n$. Assume further that $f$ is identically zero on $B(x_0,r_0)\cap \Omega^c$ for some $r_0>0$. If $0<r_1<r_0$ then for any $x \in \partial \Omega \cap B(x_0,r_1)$ and any $y \in \Omega$ 
    $$ \frac{-M}{\abs{r_0-r_1}^{\frac{sp-n}{p-1}}}\abs{x-y}^{\frac{sp-n}{p-1}} \leq \underline{P}f(y)   \leq \overline{P}f(y) \leq \frac{M}{\abs{r_0-r_1}^{\frac{sp-n}{p-1}}}\abs{x-y}^{\frac{sp-n}{p-1}}.$$
\end{lemma}
\begin{proof}
    Consider $x \in \partial \Omega \cap B(x_0,r_1)$ and let 
    $$u(z)= \frac{M}{\abs{r_0-r_1}^{\frac{sp-n}{p-1}}}\abs{z-x}^{\frac{sp-n}{p-1}}.$$
    By Proposition \ref{prop:barrier} $u$ is $(s,p)$-harmonic in $\R^n \setminus \lbrace x \rbrace $ and in particular $u$ is $(s,p)$-harmonic in $\Omega$. As $u\geq 0$ and for any $z \in \Omega^c \cap B(x_0,r_0)$, $ f(z)=0$ we have 
    \begin{equation}\label{eq:uniform-barrier-1}
    u(z ) \geq f(z), \quad \text{for all} \quad z \in \Omega^c \cap B(x_0,r_0) .
    \end{equation}
    For any $z \in B(x_0,r_0)^c$ by the triangle inequality we have $\abs{z-x} \geq r_0-r_1$. Hence,
    \begin{equation}\label{eq:uniform-barrier-2}
    u(z )\geq M \geq f(z), \quad \text{for} \quad z \in  B(x_0,r_0)^c.
    \end{equation}
    Now we choose a function $v \in \Ucal_f$. For example we consider $v(z)= M\chi_{\Omega}(z) + f(z) \chi_{\Omega^c}(z)$. We claim that
    $$w(z):= \min \lbrace u(z), v(z) \rbrace,$$
    belongs to the upper class $\Ucal_f$. As a minimum of two $(s,p)$-superharmonic functions, $w$ is $(s,p)$-superharmonic in $\Omega$. By \eqref{eq:uniform-barrier-1} and \eqref{eq:uniform-barrier-2} $u$ is above $f$ in $\Omega^c$, and since $v$ is equal to $f$ in $\Omega^c$, we have
    $$w(z)=f(z) \quad \text{for all} \quad z \in \Omega^c.$$
    Moreover, $w$ is bounded from below in $\Omega$, in fact it is nonnegative in $\Omega$. And finally for any $z \in \partial \Omega $
    \[
    \begin{aligned}
    \liminf_{\Omega \ni \xi \to z} w(\xi) &= \min \left\lbrace \liminf_{\Omega \ni \xi \to z} v(\xi) , \liminf_{\Omega \ni \xi \to z} u(\xi)  \right \rbrace \\
    &= \min \lbrace M , u(z) \rbrace = \begin{cases}
        M, \quad &\text{if} \quad z \in B(x_0,r_0)^c \cap \partial \Omega \\
        u(z), \quad &\text{if} \quad z \in B(x_0,r_0) \cap \partial \Omega.
    \end{cases} 
    \end{aligned}
    \]
    As $f(z) \leq M $ for all $z \in \R^n$ and by assumption $f$ vanishes in $B(x_0,r_0)\cap \Omega^c$
    \[
    \limsup_{\Omega^c \ni \xi \to z} f(\xi) \leq \begin{cases}
       M, \quad &\text{if} \quad z \in B(x_0,r_0)^c \cap \partial \Omega \\
        0 \leq u(z), \quad &\text{if} \quad z \in B(x_0,r_0) \cap \partial \Omega. 
    \end{cases}
    \]
    Hence, we have verified that $\liminf_{\Omega \ni \xi \to z} w(\xi) \geq  \limsup_{\Omega^c \ni \xi \to z} f(\xi) $. Thus, $w \in \Ucal_f$ and therefore, 
    $$\overline{P}f(y) \leq w(y) \leq u(y), \quad \text{for all} \quad y \in \Omega. $$
    By a similar argument $\max \left \{-u(z), -M \chi_{\Omega}(z) + f(z) \chi_{\Omega^c} \right\}$ belongs to the lower class $\Lcal_f$ and we obtain
    $$-u(y) \leq \underline{P}f(y), \quad \text{for all} \quad y \in \Omega.$$
    \end{proof}
 \subsection{Proof of Theorem \ref{thm:limit-at-infty}}

\begin{proof}
In view of Proposition \ref{prop:uniqueness}, without loss of generality, we may assume that $x_0= \mathrm{e}_n$, $y_0= -\mathrm{e}_n$, and $u(\mathrm{e}_n)= -u(-\mathrm{e}_n)= 1 $. By Propositions \ref{cor:antisymmetry} and \ref{prop:pointwise-bound}, $-1\leq u(x) \leq 1$ and 
\[
u(x)= 0, \quad \text{on } \lbrace x_n=0 \rbrace.
\]
Furthermore, by Lemma \ref{lm:equation},
\[
(-\Delta_p)^s u = 0 \quad \text{in } \R^n \setminus \lbrace \mathrm{e}_n , -\mathrm{e}_n \rbrace.
\]
Consider the following rescaled functions:
$$v_t(x) = u(tx).$$
This family is uniformly bounded. By Theorem \ref{thm:uniform-holder} it is also uniformly equi-continuous on compact subsets of $\R^n \setminus \lbrace 0 \rbrace$. For any sequence $t_j$ converging to infinity, using the Arzela-Ascoli theorem, we can pass to a subsequence such that
\[
v_{t_j} \to v_\infty \quad \text{as } t_j \to \infty, \quad \text{locally uniformly in } \R^n\setminus\lbrace 0 \rbrace
\]
\textbf{Claim:} For all such convergent subsequences,  the limit $v_\infty$ is identically zero.

Once the claim is proved, we can verify that $u$ converges to zero at infinity. If that is not the case, there should exist, $\delta >0$ a sequence $x_j$ such that $\lim_{j \to \infty}\abs{x_j}=\infty$ and 
$$\abs{u(x_j)} >\delta.$$
This means that 
\[
\left| v_{\abs{x_j}}\left( \frac{x_j}{\abs{x_j}} \right) \right | > \delta.
\]
As $\frac{x_j}{\abs{x_j}} \in \partial B(0,1)$, after passing to a subsequence we can find $x_\infty \in \partial B(0,1)$ such that
\[
\lim_{j \to \infty } x_j = x_\infty.
\]
As mentioned before, $v_{\abs{x_j}}$ are uniformly bounded and equi-continuous on compact subsets of $\R^n \setminus \lbrace 0 \rbrace$. Using the Arzela-Ascoli theorem, we can pass to a subsequence again so that $v_{\abs{x_j}}$ converges locally uniformly to a limit $v_\infty$ on $\R^n \setminus \lbrace 0 \rbrace$. Hence,
\[
\delta \leq\lim_{j \to \infty} \left |v_{\abs{x_j}}\left( \frac{x_j}{\abs{x_j}}\right) \right| = v_\infty(x_\infty).
\]
This is in contradiction with our claim. To finish the proof, it only remains to verify the claim. 
\\
\textbf{Proof of the claim} Notice that $v_\infty$ is zero on $\lbrace x_n =0 \rbrace \setminus \lbrace 0 \rbrace$. As $v_j$ converges locally uniformly to $v_\infty$, by the stability property of viscosity solutions, $v_\infty$ is a viscosity solution of
$$(-\Delta_p)^s v_\infty = 0 \quad \text{ in } \R^n \setminus \lbrace 0 \rbrace .$$
The strategy is to show that $v_\infty$ is a solution in the whole $\R^n$. Then an application of the Liouville theorem implies that $v_\infty$ is constant. The main challenge is to show that $v_\infty$ has a limit at the origin.
\\
\textit{Step 1.} We compare $v_\infty$ with the Perron solutions of appropriate complementary values in the domain
$$\Omega := B_1 \setminus \bigl\lbrace \lbrace x_n =x_{n-1}=x_{n-2}= ... = x_2=0 \rbrace \cap \lbrace x_{1} \geq 0 \rbrace \bigr\rbrace. $$

Let $U$ be the following function defined on $\Omega^c$:
\[
\begin{aligned}
 U&= 1 \text{ in } B_1^c \\
 U&= 0 \text{ on } B_1 \cap  \left\lbrace \lbrace x_n =x_{n-1}=x_{n-2}= ... = x_2=0 \rbrace \cap \lbrace 0< x_{1} < 1 \rbrace \right\rbrace \\
 U&= 1 \text{ at } 0 .
 \end{aligned}
 \]
Similarly, we define $L$ by:
\[
\begin{aligned}
 L&= -1 \text{ in } B_1^c \\
 L&= 0 \text{ on } \left\lbrace \lbrace x_n =x_{n-1}=x_{n-2}= ... = x_2=0 \rbrace \cap \lbrace 0< x_{1} < 1 \rbrace \right\rbrace \\
 L&= -1 \text{ at } 0 .
 \end{aligned}
 \]
 Let $h$ be an arbitrary function in $\Ucal_U$. Observe that 
 $$v_\infty(x) \leq 1 \leq h(x) \quad \text{for almost every} \quad   x \in \R^n \setminus \Omega.$$
 At every boundary point $x \in \partial \Omega$
\begin{equation}\label{eq:U-uppersemicontinuity}
\begin{aligned}
 \limsup_{\R^n \setminus \Omega \ni y \to x} &U(y)= U(x)\\
 &=
\begin{cases}
   0 \quad \text{if} \quad x \in B_1\cap  \lbrace x_n =x_{n-1}= ... = x_2=0 \rbrace \cap \lbrace 0< x_{1} < 1 \rbrace  , \\
   1\quad \text{if}\quad x \in \partial B_1 \cup \lbrace 0 \rbrace .
\end{cases}
\end{aligned}
\end{equation}
 Therefore, at every $x \in \partial \Omega$
\[
\liminf_{\Omega \ni y \to x} h(y) \geq\limsup_{\R^n \setminus \Omega \ni y \to x} U(y) =U(x) \geq \limsup_{\Omega \ni y \to x } v_\infty(y).
\]
In the last inequality, we have used that $v_\infty$ is continuous outside of the origin, $v_\infty= 0$ on $\lbrace x_n =0\rbrace \setminus \lbrace 0 \rbrace$, and the bound $-1 \leq v_\infty \leq 1$. As $v_\infty$ is $(s,p)$-harmonic in $\Omega$, by the comparison principle 
 $$v_\infty(x) \leq h(x), \quad x \in \Omega. $$
 Hence, taking infimum over all $h \in \Ucal_U$, 
 $$v_\infty(x) \leq \overline{P}U(x), \quad x \in \Omega.$$
 Similarly,
 $$v_\infty(x) \geq  \underline{P}L(x), \quad x \in \Omega.$$

 We claim that $\overline{P}U$ and $\underline{P}L$ have zero limits at the origin. Before doing that, let us demonstrate some pointwise bounds for $\overline{P}U$.
 
Take $h \in \Ucal_U$, by the comparison we have $h(x)\geq 0$. Hence, $\overline{P}U(x) \geq 0$. As 
$$\chi_{\Omega} +  U(x)\chi_{\R^n \setminus \Omega} \in \Ucal_U,  $$
we have $\overline{P}U(x) \leq 1$  for every $x \in \Omega$ .
\\
\textit{Step 2.} We show that $\overline{P}U$ has radial limits at the origin. For every $0< \rho < 1$ define 
$$V_\rho(x):= \overline{P}U(\rho x).$$
Notice that for every $0<\rho < 1$ we have $V_\rho \leq U$ on $\Omega^c$ and $V_\rho$ is $(s,p)$-harmonic inside $\Omega$. Since $\overline{P}U$ is bounded between zero and one, we have
$$0 \leq V_\rho\leq 1 .$$
By Lemma \ref{lm:barrier}   
$$\lim_{\Omega \ni y \to x} V_\rho( y) =0 \quad \text{for } x \in B_1 \cap \left\lbrace \lbrace x_n =x_{n-1}=x_{n-2}= ... = x_2=0 \rbrace \cap \lbrace 0< x_{1} < \frac{1}{\rho} \rbrace \right\rbrace .$$
Hence, if $0< \rho <1$, by \eqref{eq:U-uppersemicontinuity}, for any $h \in \Ucal_U$ at any point $x \in \partial \Omega$ 
\[
\liminf_{\Omega \ni y \to x} h(y) \geq U(x) \geq \limsup_{\Omega \ni y \to x} V_\rho[] y).
\]
Therefore, by the comparison principle $h(x) \geq V_\rho(x)$ for every $x \in \Omega$,. Hence
$$V_\rho(x) \leq \overline{P}U(x)$$
In particular, this implies that for $x \in \Omega$ the map $\rho \to V_\rho(x)$ is non-decreasing for $\rho \in (0,1)$. Indeed, for $1>\rho_1 >\rho_2>0$ we have
$$V_{\frac{\rho_2}{\rho_1}}(x) \leq \overline{P}U(x) \quad \text{ in } \Omega.$$
Therefore, for $x \in \Omega$ we have
$$V_{\rho_2}(x)=\overline{P}U(\rho_2 x) = V_{\frac{\rho_2}{\rho_1}}(\rho_1 x) \leq \overline{P}U(\rho_1 x)= V_{\rho_1}(x).$$
Hence, the radial limits exist for the function $\overline{P}U$. Let $V_0(x):= \lim_{\rho \to 0} V_\rho(x)$. As a radial limit of a function, $V_0$ is zero-homogeneous, that is, there exists a function $g: \Sbb^{n-1} \to \R$ such that
\begin{equation}\label{eq:limit-zero-homogen}
    V_0(x)=g\left(\frac{x}{\abs{x}} \right), \quad \text{for} \quad x \neq 0.
\end{equation}
\textit{Step 3.} We now show that $\underset{x \to 0}{\lim}\, \overline{P}U=0$. To prove this, we argue towards a contradiction. Let us assume there exists a sequence $x_i\in \Omega$ such that $x_i$ converges to the origin, but $\underset{i \to \infty}{\limsup}\, \overline{P}U(x_i) \neq 0$. By passing to a subsequence, we may further assume that $\frac{1}{2} >\abs{x_i} > \abs{x_{i+1}}$ for every $i\in \mathbb{N}$, and that there exists a $\delta > 0$ such that $\left| \overline{P}U(x_i)\right | > \delta$. We consider the functions $V_{\abs{x_i}}$.  Since $V_{\abs{x_i}}$ is a uniformly bounded sequence ( $0 \leq V_{\abs{x_i}}\leq 1$) and  $V_{\abs{x_i}}$ is an $(s,p)$-harmonic function in  $ B(0, \frac{1}{\abs{x_i}}) \setminus \left\lbrace \lbrace x_n =x_{n-1}=x_{n-2}= ... = x_2=0\rbrace \cap \lbrace x_{1} \geq 0 \rbrace \right\rbrace$, we can pass to a subsequence, such that $V_{\abs{x_i}}$ converges locally uniformly to $V_0$ in \\ $\R^n \setminus\left\lbrace \lbrace x_n =x_{n-1}=x_{n-2}= ... = x_2=0\rbrace \cap \lbrace x_{1} \geq 0 \rbrace \right\rbrace$. By the stability property of viscosity solutions, $V_0$ is  $(s,p)$-harmonic in $\R^n \setminus\left\lbrace \lbrace x_n =x_{n-1}=x_{n-2}= ... = x_2=0\rbrace \cap \lbrace x_{1} \geq 0 \rbrace \right\rbrace$.  

Using Lemma \ref{lm:barrier} 
 with $x_0= \abs{x_i} \mathrm{e}_1$ and $r_0=\frac{3}{2}r_1= \abs{x_i}$, the functions $V_{\abs{x_i}}$ have a uniform modulus of continuity on $ \lbrace x_n =x_{n-1}=x_{n-2}= ... = x_2=0\rbrace \cap \lbrace \frac{1}{2}< x_{1} <\frac{3}{2} \rbrace$. More precisely given $x \in \lbrace x_n =x_{n-1}=x_{n-2}= ... = x_2=0\rbrace \cap \lbrace \frac{1}{2}< x_{1} <\frac{3}{2} \rbrace$ for any $y$ in $B_{1/6}(x)$ we have  
\[
\abs{V_{\abs{x_i}}(y)}=\abs{ \overline{P}U(y\abs{x_i})} \leq \frac{1}{(\abs{x_i}- \frac{2}{3}\abs{x_i})^{\frac{sp-n}{p-1}}} \left|(y-x)\abs{x_i}\right|^\frac{sp-n}{p-1} = 3^\frac{sp-n}{p-1} \abs{y-x}^\frac{sp-n}{p-1}.
\]
Hence, the convergence $V_{\abs{x_i}} \to V_0$ is localy uniform in $B(0,3/2)\setminus B(0,1/2)$, and $V_0$ is continuous in $B(0,3/2)\setminus B(0,1/2)$. Therefore, by \eqref{eq:limit-zero-homogen}, $g$ is continuous on $\Sbb^{n-1}$. Now we demonstrate that $g \equiv 0$. Recall that $g(\mathrm{e}_1)= v_0(\mathrm{e}_1)=0$. As a continuous function, the maximum and minim of $g$ is achieved on $\Sbb^{n-1}$. If $g$ is not constant either the minimum or the maximum of $g$ is non zero, and hence it can not be achieved at $\mathrm{e}_1$. Assume that the minimum of $g$ is achieved at $\tilde{x} \neq \mathrm{e}_1$. As $V_0$ is zero-homogeneous, $V_0(\tilde{x})$ is the essential minimum of $V_0$. By Lemma \ref{lm:viscosity-evaluation}
\begin{equation}\label{eq:limit-v_0-evaluation}
0 \leq  \int_{\R^n} \frac{V(\tilde{x})-V_0(y)}{\abs{\tilde{x}-y}^{n+sp}} \dd y . 
\end{equation}
Since, $V(\tilde{x}) \leq V(y)$ for all $y \neq 0$, \eqref{eq:limit-v_0-evaluation} implies that $V_0$ is constant outside of the origin. As $V_0$ is zero on the half line $\lbrace x_n =x_{n-1}=x_{n-2}= ... = x_2=0\rbrace \cap \lbrace x_{1} > 0 \rbrace$, we must have 
$$V_0(x)=0 \quad \text{for} \quad x \neq 0.$$
    Now by assumption, we have
$$\overline{P}U(x_i)= \left|V_{\abs{x_i}}\left(\frac{x_i}{\abs{x_i}} \right) \right| > \delta$$
    After passing to a subsequence, we may assume that 
    $$\lim_{i \to \infty} \frac{x_i}{\abs{x_i}} =x_0 \in \Sbb^{n-1}. $$
    But the  uniform convergence of $V_{\abs{x_1}} $ to $V_0$ in $B(0,2) \setminus B(0, \frac{1}{2})$ leads to 
    $$\lim_{i \to \infty} V_{\abs{x_i}}\left(\frac{x_i}{\abs{x_i}} \right) = V_0(x_0)=0. $$
    This contradiction shows that the assumption about $\overline{P}U(x_i)$ was wrong and hence,  
    $$\lim_{\Omega \ni x \to 0} \overline{P}U(x)=0 .$$
    In a similar way, we can show that 
    $$\lim_{\Omega \ni x \to 0} \underline{P}L(x)=0.$$
    As $v_\infty$ is trapped between $\overline{P}U$ and $\underline{P}L$ we arrive at
    $$\lim_{ x \to 0} v_{\infty}(x)=0.$$
    Hence, modifying the value of $v_\infty$ at the origin so that $v_\infty(0)=0$, we have established the continuity of $v_\infty$ at the origin, in the next step, we show that $v_\infty$ is $(s,p)$-harmonic in the whole $\R^n$.\\
    \textit{Step 4.} We already know that $v_\infty$ is a viscosity solution of 
    $$(-\Delta_p)^s v_\infty(x) = 0, \quad \text{in} \quad \R^n \setminus \lbrace 0 \rbrace.$$
    We use the anti-symmetry of $v_\infty$ to verify the equation at the origin. We have
    \begin{equation}\label{eq:limit-antisymmetry-PV}
    \begin{aligned}
    (-\Delta_p)^s v_\infty(0 )= \mathrm{P.V.} \int_{\R^n} \frac{J_p(v_\infty(0) - v_\infty(y))}{\abs{0-y}^{n+sp}} \dd y = \lim_{\epsilon \to 0 } \int_{\R^n \setminus B(0,\epsilon)} \frac{-J_p(v_\infty(y))}{\abs{y}^{n+sp}} \dd y = 0.
    \end{aligned}
    \end{equation}
    The last equality is due to the anti-symmetry of $v_\infty$. 
    Now, we formally verify that test functions touching $v_\infty$ at the origin satisfy the equation. 
    As $n \geq 2$ and $sp>n \geq 2$ we are in the range $p > \frac{2}{2-s}$. Consider $\phi \in C^2(B(0,r))$ for some $r>0$, such that $\phi(0)=v_\infty(0)=0$ and 
    $$\phi(x) \leq v_\infty(x) \quad \text{for} \quad x \in B(0,r).$$
    Let 
    \[
    w(x):= \begin{cases}
        \phi(x) \quad &\text{for} \quad x \in B(0,r), \\
        v_\infty (x) \quad &\text{for} \quad x \in \R^n \setminus B(0,r).
    \end{cases}
    \]
    Then, using \eqref{eq:limit-antisymmetry-PV}
    \[
    \begin{aligned}
    \mathrm{P.V.} \int_{\R^n} \frac{J_p(w(0) - w(y))}{\abs{0-y}^{n+sp}} \dd y &= \mathrm{P.V.} \int_{\R^n} \frac{-J_p(  w(y))}{\abs{y}^{n+sp}} \dd y \\
    & \geq \mathrm{P.V.} \int_{\R^n} \frac{-J_p(  v_{\infty}(y))}{\abs{y}^{n+sp}} \dd y = 0.
    \end{aligned}
    \]
    Similarly, for a test function $\phi$ touching $v_\infty$ from above at the origin, define
    \[
    w(x):= \begin{cases}
        \phi(x) \quad &\text{for} \quad x \in B(0,r), \\
        v_\infty (x) \quad &\text{for} \quad x \in \R^n \setminus B(0,r).
    \end{cases}
    \]
    We easily verify
    $$\mathrm{P.V.} \int_{\R^n} \frac{J_p(w(0) - w(y))}{\abs{0-y}^{n+sp}} \dd y \leq 0.$$
    \textit{Step 5.} Conclusion. Since we have established that $v_\infty$ is $(s,p)$-harmonic in the whole $\R^n$, by Liouville's theorem $v_\infty$ is constant. As $v_\infty$ is zero on the hyperplane $\lbrace x_n=0 \rbrace$, $v_\infty$ is identically zero and the claim is proved. 
\end{proof}

\section{Non-vanishing of extremals in the half-spaces}
Here we argue that the extremals have a sign in each half-space above and below the affine hyperplane of anti-symmetry. 
The argument is based on a maximum principle for anti-symmetric functions. See for example \cite{CeL}.
\begin{proposition}
    Let $u$ be a non-constant Morrey extremal with
    $$[u]_{C^{s-\frac{n}{p}}}= \frac{\abs{u(x_0) - u(y_0)}}{\abs{x_0-y_0}^{s-\frac{n}{p}}}.$$
    Then $u(x)- \frac{1}{2} \left( u(x_0) + u(y_0) \right)$ does not change sign (and does not vanish) in the regions 
    \[
    \Sigma^+ :=\left \lbrace x\in \R^n \; : \; \left(x-\frac{1}{2}(x_0+y_0)\right)\cdot (x_0-y_0)>0   \right \rbrace ,
    \]
    and
    \[
    \Sigma^- :=\left \lbrace x\in \R^n \; : \; \left(x-\frac{1}{2}(x_0+y_0)\right)\cdot (x_0-y_0)<0   \right \rbrace ,
    \]
\end{proposition}
\begin{proof}
    Without loss of generality, we may assume $x_0=\mathrm{e}_n$, $y_0=-\mathrm{e}_n$, and $u(\mathrm{e}_n)=-u(-\mathrm{e}_n)=1$. We prove that $u$ is positive in 
    $$\Sigma^+ = \left \lbrace x \in \R^n \, : \, x \cdot \mathrm{e}_n>0  \right \rbrace.$$
    The argument for $u$ being negative in $\Sigma^-$ is similar. 
    From Proposition \ref{cor:antisymmetry} we know that $u$ is anti-symmetric with respect to 
    $$\Pi= \left\lbrace  x \in \R^n \, :\, x \cdot \mathrm{e}_n=0 \right \rbrace.$$
     Now consider
    \[
    \lambda^+= \inf_{x \in \Sigma^+} u(x).
    \]
    Since $u$ tends to zero at infinity and vanishes on $\Pi= \partial \Sigma^+$, 
    $$-\infty < \lambda^+ \leq 0.$$
    Our claim that $u$ is non-negative is equivalent to showing that $\lambda^+=0$. Note that since $u$ tends to zero at infinity, if $\lambda^+<0$ then there should exist a point $z_0 \in \Sigma^+$ such that
    $$u(z_0) = \lambda^+ .$$
    If $\lambda^+=0$, our claim is that $u(x)>0$, for all $x\in \Sigma^+$. Thus, to verify our claim, it is enough to show that there is no point $z_0 \in \Sigma^+$ so that $u(z_0)= \lambda^+$. For the sake of contradiction assume that 
    $$u(z_0)= \min_{x\in \Sigma^+} u(x).$$
    Obviously $z_0 \neq \mathrm{e}_n$. Now we use the equation for $u$. Recall that $u$ is a viscosity solution of 
    $$(-\Delta_p)^s u(x) =0 \quad \text{for}\quad x \in \R^n \setminus \lbrace \mathrm{e}_n,-\mathrm{e}_n \rbrace.$$
    As $u(z_0)$ is a local minimum for $u$, Lemma \ref{lm:viscosity-evaluation} implies
    \[
    \int_{\R^n} \frac{J_p(u(z_0)- u(y))}{\abs{z_0-y}^{n+sp}} \dd y \geq 0.
    \]
    We split the integral into two parts. For any $y \in \R^n$ we define $\tilde{y}:= y-2 y\cdot \mathrm{e}_n$. With this notation we have
    \begin{equation}\label{eq:nonvanish-split}
    \begin{aligned}
        0 \leq \int_{\R^n} \frac{J_p(u(z_0)- u(y))}{\abs{z_0-y}^{n+sp}} \dd y  &= \int_{\Sigma^+}  \frac{J_p(u(z_0)- u(y))}{\abs{z_0-y}^{n+sp}} \dd y  + \int_{\Sigma^-} \frac{J_p(u(z_0)- u(y))}{\abs{z_0-y}^{n+sp}} \dd y  \\
        &=  \int_{\Sigma^+} \frac{J_p(u(z_0)- u(y))}{\abs{z_0-y}^{n+sp}} \dd y  + \int_{\Sigma^+} \frac{J_p(u(z_0)- u(\tilde{y}))}{\abs{z_0-\tilde{y}}^{n+sp}} \dd y \\
        &= \int_{\Sigma^+} \frac{J_p(u(z_0)- u(y))}{\abs{z_0-y}^{n+sp}} \dd y  - \int_{\Sigma^+} \frac{J_p(-u(z_0)+ u(\tilde{y}))}{\abs{z_0-\tilde{y}}^{n+sp}} \dd y \\
        \text{(using anti-symmetry)}\quad&= \int_{\Sigma^+} \frac{J_p(u(z_0)- u(y))}{\abs{z_0-y}^{n+sp}} \dd y - \int_{\Sigma^+} \frac{J_p(-u(z_0) - u(y))}{\abs{z_0-\tilde{y}}^{n+sp}} \dd y \\
        &= \int_{\Sigma^+} J_p(u(z_0)- y(y))\left( \frac{1}{\abs{z_0-y}^{n+sp}}   - \frac{1}{\abs{z_0-\tilde{y}}^{n+sp}}\right) \dd y \\
        &\quad + \int_{\Sigma^+} \frac{J_p(u(z_0)- u(y)) - J_p(-u(z_0) - u(y))}{\abs{z_0-\tilde{y}}^{n+sp}} \dd y \\
        &= \int_{\Sigma^+} J_p(u(z_0)- y(y))\left( \frac{1}{\abs{z_0-y}^{n+sp}}   - \frac{1}{\abs{z_0-\tilde{y}}^{n+sp}}\right) \dd y \\
        &\quad + \int_{\Sigma^+} \frac{J_p(u(z_0)- u(y)) - J_p(-u(z_0) - u(y))}{\abs{\tilde{z}_0-y}^{n+sp}} \dd y \\
        &:= \Ical_1 + \Ical_2
    \end{aligned}
    \end{equation}
    Now we aim to show that $\Ical_1 <0$ and $\Ical_2 \leq0$ so that we reach a contradiction. 
    First, notice that as $z \in \Sigma^+$, 
    $$z_0\cdot(y-\tilde{y}) = 2\abs{y-\tilde{y}} z_0\cdot \mathrm{e}_n >0.$$
    Hence,
    \[
    \abs{z_0 - y}^2 = \abs{z_0}^2 + \abs{y}^2 - 2 z_0 \cdot y < \abs{z_0}^2 + \abs{\tilde{y}}^2 - 2 z_0 \cdot \tilde{y} = \abs{z_0-\tilde{y}}^2 .
    \]
    Therefore, 
    \[
    \frac{1}{\abs{z_0-y}^{n+sp}}   - \frac{1}{\abs{z_0-\tilde{y}}^{n+sp}} >0 .
    \]
    As $u(z_0)$ is the minimum of $u$ on $\Sigma^+$, 
    $$J_p(u(z_0)-u(y)) \leq 0 \quad \text{for } y \in \Sigma^+ . $$
    Since $u$ is not constant on $\Sigma^+$ we arrive at the strict inequality
    \[
    \Ical_1= \int_{\Sigma^+} J_p(u(z_0)- y(y))\left( \frac{1}{\abs{z_0-y}^{n+sp}}   - \frac{1}{\abs{z_0-\tilde{y}}^{n+sp}}\right) \dd y <0.
    \]
    As
    \[
    0\geq \lambda^+=u(z_0),
    \]
    we have
    \[
    u(z_0)-u(y) \leq -u(z_0)-u(y).
    \]
    Since $J_p(t)$ is a monotone function, we get 
    \[
    J_p(u(z_0)- u(y)) - J_p(-u(z_0) - u(y)) \leq 0.
    \]
    Hence,
     \[
    \Ical_2=  \int_{\Sigma^+} \frac{J_p(u(z_0)- u(y)) - J_p(-u(z_0) - u(y))}{\abs{\tilde{z_0}-y}} \dd y \leq 0.
     \]
\end{proof}



\providecommand{\bysame}{\leavevmode\hbox to3em{\hrulefill}\thinspace}
\providecommand{\MR}{\relax\ifhmode\unskip\space\fi MR }
\providecommand{\MRhref}[2]{%
  \href{http://www.ams.org/mathscinet-getitem?mr=#1}{#2}
}
\providecommand{\href}[2]{#2}

\end{document}